\documentclass[11pt,reqno,british]{amsart}
\usepackage{charter}
\usepackage[T1]{fontenc}
\usepackage[latin9]{inputenc}
\usepackage{geometry}
\usepackage{a4wide}
\usepackage{color}
\usepackage{prettyref}
\usepackage{mathrsfs}
\usepackage{enumitem}
\usepackage{amstext}
\usepackage{amsthm}
\usepackage{amssymb}
\usepackage{setspace}
\usepackage{wasysym}
\usepackage[numbers]{natbib}
\usepackage[unicode=true,
 bookmarks=true,bookmarksnumbered=false,bookmarksopen=false,
 breaklinks=false,pdfborder={0 0 1},backref=false,colorlinks=true]
 {hyperref}
\hypersetup{
 pdfauthor={Adam Skalski and Ami Viselter},
 pdfstartview={XYZ null null 1.25},citecolor=OliveGreen,linkcolor=RoyalBlue,urlcolor=blue}

\makeatletter
\numberwithin{equation}{section}
\theoremstyle{remark}
\newtheorem*{convention*}{Convention}
\theoremstyle{plain}
\newtheorem{thm}{\protect\theoremname}[section]
\theoremstyle{definition}
\newtheorem{defn}[thm]{\protect\definitionname}
\theoremstyle{plain}
\newtheorem{prop}[thm]{\protect\propositionname}
\theoremstyle{plain}
\newtheorem{lem}[thm]{\protect\lemmaname}
\theoremstyle{remark}
\newtheorem{rem}[thm]{\protect\remarkname}
\theoremstyle{plain}
\newtheorem{cor}[thm]{\protect\corollaryname}
\theoremstyle{definition}
\newtheorem{example}[thm]{\protect\examplename}
\theoremstyle{remark}
\newtheorem{notation}[thm]{\protect\notationname}

\usepackage{eucal}
\let\mathcal=\CMcal
\usepackage{dsfont}
\usepackage{graphicx}
\usepackage{mathtools} 
\usepackage[all]{xy}
\usepackage[usenames,dvipsnames]{xcolor}
\usepackage{upgreek}

\let\ldash=\l

\newrefformat{eq}{\eqref{#1}}
\newrefformat{def}{Definition~\ref{#1}}
\newrefformat{lem}{Lemma~\ref{#1}}
\newrefformat{pro}{Proposition~\ref{#1}}
\newrefformat{prop}{Proposition~\ref{#1}}
\newrefformat{cor}{Corollary~\ref{#1}}
\newrefformat{thm}{Theorem~\ref{#1}}
\newrefformat{exa}{Example~\ref{#1}}
\newrefformat{rem}{Remark~\ref{#1}}
\newrefformat{sec}{Section~\ref{#1}}
\newrefformat{sub}{Subsection~\ref{#1}}
\newrefformat{conj}{Conjecture~\ref{#1}}
\newrefformat{desc}{\ref{#1}}
\newrefformat{ques}{Question~\ref{#1}}

\DeclareSymbolFont{YHlargesymbols}{OMX}{yhex}{m}{n}
\DeclareMathAccent{\wideparen}{\mathord}{YHlargesymbols}{"F3}

\makeatother

\providecommand{\corollaryname}{Corollary}
\providecommand{\definitionname}{Definition}
\providecommand{\examplename}{Example}
\providecommand{\lemmaname}{Lemma}
\providecommand{\notationname}{Notation}
\providecommand{\propositionname}{Proposition}
\providecommand{\remarkname}{Remark}
\providecommand{\theoremname}{Theorem}

\begin{document}
\global\long\def\e{\varepsilon}%
\global\long\def\N{\mathbb{N}}%
\global\long\def\Z{\mathbb{Z}}%
\global\long\def\Q{\mathbb{Q}}%
\global\long\def\R{\mathbb{R}}%
\global\long\def\C{\mathbb{C}}%
\global\long\def\G{\mathbb{G}}%
\global\long\def\QG{\mathbb{G}}%
\global\long\def\QH{\mathbb{H}}%
\global\long\def\bn{\mathbb{N}}%
\global\long\def\br{\mathbb{R}}%
\global\long\def\bc{\mathbb{C}}%
\global\long\def\bt{\mathbb{T}}%

\global\long\def\H{\EuScript H}%
\global\long\def\J{\mathcal{J}}%
\global\long\def\K{\mathcal{K}}%
\global\long\def\KHilb{\EuScript K}%
\global\long\def\a{\alpha}%
\global\long\def\be{\beta}%
\global\long\def\l{\lambda}%
\global\long\def\om{\omega}%
\global\long\def\z{\zeta}%
\global\long\def\gnsmap{\upeta}%
\global\long\def\Aa{\mathcal{A}}%
\global\long\def\Aalg{\mathsf{A}}%
\global\long\def\Sant{\mathtt{S}}%
\global\long\def\Rant{\mathtt{R}}%

\global\long\def\Ree{\operatorname{Re}}%
\global\long\def\Img{\operatorname{Im}}%
\global\long\def\linspan{\operatorname{span}}%
\global\long\def\supp{\operatorname{supp}}%
\global\long\def\slim{\operatorname*{s-lim}}%
\global\long\def\clinspan{\operatorname{\overline{span}}}%
\global\long\def\co{\operatorname{co}}%
\global\long\def\pres#1#2#3{\prescript{#1}{#2}{#3}}%

\global\long\def\tensor{\otimes}%
\global\long\def\tensormin{\mathbin{\otimes_{\mathrm{min}}}}%
\global\long\def\tensorn{\mathbin{\overline{\otimes}}}%

\global\long\def\A{\forall}%

\global\long\def\vNa{M}%
\global\long\def\matrices{\mathrm{M}}%

\global\long\def\i{\mathrm{id}}%
\global\long\def\Tr{\mathrm{Tr}}%

\global\long\def\one{\mathds{1}}%
\global\long\def\Ww{\mathds{W}}%
\global\long\def\wW{\text{\reflectbox{\ensuremath{\Ww}}}\:\!}%
\global\long\def\op{\mathrm{op}}%
\global\long\def\WW{{\mathds{V}\!\!\text{\reflectbox{\ensuremath{\mathds{V}}}}}}%
\global\long\def\Vv{\mathds{V}}%
\global\long\def\vV{\text{\reflectbox{\ensuremath{\Vv}}}\:\!}%

\global\long\def\M#1{\operatorname{M}(#1)}%
\global\long\def\Linfty#1{L^{\infty}(#1)}%
\global\long\def\Lone#1{L^{1}(#1)}%
\global\long\def\Lp#1{L^{p}(#1)}%
\global\long\def\Lq#1{L^{q}(#1)}%
\global\long\def\LoneSharp#1{L_{\sharp}^{1}(#1)}%
\global\long\def\Ltwo#1{L^{2}(#1)}%
\global\long\def\Cz#1{\mathrm{C}_{0}(#1)}%
\global\long\def\CzU#1{\mathrm{C}_{0}^{\mathrm{u}}(#1)}%
\global\long\def\CzUSSharp#1{\mathrm{C}_{0}^{\mathrm{u}}(#1)_{\sharp}^{*}}%
\global\long\def\CU#1{\mathrm{C}^{\mathrm{u}}(#1)}%
\global\long\def\Cb#1{\mathrm{C}_{b}(#1)}%
\global\long\def\CStarF#1{\mathrm{C}^{*}(#1)}%
\global\long\def\CStarR#1{\mathrm{C}_{\mathrm{r}}^{*}(#1)}%
\global\long\def\Cc#1{\mathrm{C}_{c}(#1)}%
\global\long\def\CC#1{\mathrm{C}(#1)}%
\global\long\def\CzX#1#2{\mathrm{C}_{0}^{#1}(#2)}%
\global\long\def\CcX#1#2{\mathrm{C}_{c}^{#1}(#2)}%
\global\long\def\CcTwo#1{\mathrm{C}_{c}^{2}(#1)}%
\global\long\def\aa#1{\mathfrak{a}_{#1}}%

\global\long\def\linfty#1{\ell^{\infty}(#1)}%
\global\long\def\lone#1{\ell^{1}(#1)}%
\global\long\def\ltwo#1{\ell^{2}(#1)}%
\global\long\def\cz#1{\mathrm{c}_{0}(#1)}%
\global\long\def\Pol#1{\mathrm{Pol}(#1)}%
\global\long\def\Ltwozero#1{L_{0}^{2}(#1)}%
\global\long\def\Irred#1{\mathrm{Irred}(#1)}%
\global\long\def\conv{\star}%

\global\long\def\Ad#1{\mathrm{Ad}(#1)}%
\global\long\def\VN#1{\mathrm{VN}(#1)}%
\global\long\def\d{\,\mathrm{d}}%
\global\long\def\t{\mathrm{t}}%

\global\long\def\tie#1{\wideparen{#1}}%

\title[Quantum generating functionals]{Generating functionals for locally compact quantum groups}
\author{Adam Skalski}
\address{Institute of Mathematics of the Polish Academy of Sciences,
	ul.~\'Sniadeckich 8, 00--656 Warszawa, Poland}
\email{a.skalski@impan.pl}

\author{Ami Viselter}
\address{Department of Mathematics, University of Haifa, 31905 Haifa, Israel}

\email{aviselter@univ.haifa.ac.il}


\begin{abstract}
Every symmetric generating functional of a convolution semigroup of states  on a locally compact quantum group is shown to admit a dense unital $*$-subalgebra with core-like properties in its domain. On the other hand we prove that every normalised, symmetric, hermitian conditionally positive functional on a dense $*$-subalgebra of the unitisation of the universal C$^*$-algebra of a locally compact quantum group, satisfying certain technical conditions, extends in a canonical way to a generating functional. Some consequences of these results are outlined, notably those related to constructing cocycles out of convolution semigroups.    		
\end{abstract}

\maketitle

\section*{Introduction}
Convolution semigroups of probability measures on a locally compact group on one hand are a source of a rich and interesting class of Markov semigroups on classical function spaces, and, on the other hand, form a fundamental notion in the study of L\'evy processes  (stochastic processes with independent and identically distributed increments). In the abstract context of measure spaces, Markov semigroups, as one-parameter semigroups of operators, are naturally studied via their generators \cite{EngelNagel}. 
The additional translation invariance of the operator semigroups coming from convolution semigroups of measures, afforded by the group structure, yields another, a priori simpler tool which determines the semigroup uniquely: the so-called generating functional, given by differentiating the measures themselves at $t=0$. The generating functional can be viewed as a `localised' version of the semigroup generator (so that in the case of the heat semigroup on $\mathbb{R}$ the generating functional evaluated at a smooth function is just its second derivative at $0$), and plays a key role in the L\'evy--Khintchine formula and its generalisations \cite{Heyer__book_prob_meas}. When the group in question is abelian, the Fourier transform allows us to view  generating functionals equivalently as conditionally negative-definite functions on the dual locally compact group. This point of view turns out to be very useful in certain approaches to  potential theory \cite{BergForst}.

Not surprisingly, generating functionals played a key role in quantum generalisations of classical convolution semigroups to the framework of compact quantum groups, or more generally $*$-bialgebras, initially developed primarily by Sch\"urmann and his collaborators \cite{Schurmann__white_noise_bialg}. Sch\"urmann's reconstruction theorem says in particular that each normalised, hermitian and conditionally positive functional on a $*$-bialgebra indeed comes from a uniquely determined convolution semigroup of states. This allows one to define and study various properties of convolution semigroups of states (for instance Gaussianity) directly via generating functionals, and was put in use with great success for example in \cite{Cipriani_Franz_Kula__sym_Levy_proc}.
If one wants to extend this study to the framework of locally compact quantum groups of Kustermans and Vaes \cite{Kustermans_Vaes__LCQG_C_star}, one encounters immediately a significant stumbling block: although in \cite{Lindsay_Skalski__conv_semigrp_states} Lindsay and the first-named author showed that each convolution semigroup of states on a locally compact quantum group admits a densely-defined generating functional, which moreover determines the semigroup uniquely, contrary to the classical (or dual to classical) and compact quantum cases there is no apparent canonical subalgebra inside the functional's domain. This means that it is far from straightforward to express properties of the semigroup via the properties of its generating functional (at least in the way it was done in the compact case) and makes it very difficult to conceive of a suitable version of the reconstruction theorem. Therefore in the predecessor of this paper, \cite{Skalski_Viselter__convolution_semigroups}, we discussed the generating functionals only briefly, and exploited the additional $L^2$-symmetry assumption, using quantum Dirichlet forms as the main tool.

In this article we keep the mentioned symmetry assumption, but revisit the matter of generating functionals. Perhaps surprisingly, it turns out that in the most general locally compact quantum group context a useful path comes again, as in the abelian situation of \cite{BergForst}, from the Fourier transform ideas, combined with the Dirichlet form techniques of \cite{Skalski_Viselter__convolution_semigroups}. Specifically, we use quantum Fourier transforms and noncommutative Dirichlet forms to realise the two aims alluded to above. Firstly, we prove that every generating functional contains in its domain a dense unital $*$-subalgebra, such that the corresponding restriction of the functional determines the semigroup uniquely. Secondly, we establish a reconstruction theorem:  under certain, somewhat complicated (but satisfied in natural examples) technical conditions a conditionally positive functional on a dense unital $*$-subalgebra of the universal C$^*$-algebra of a locally compact quantum group admits an extension to a uniquely determined generating functional of a convolution semigroup of positive functionals. As we indicate in the text, the results of this type for example allow us to associate to every convolution semigroup of states as above certain canonically defined cocycles. At the same time the key task of finding a common dense domain for all generating functionals associated with a given locally compact quantum group remains for now beyond our reach.  

The contents of the paper are as follows. After recalling some preliminary facts and notations in Section \ref{sec:prelims}, in Section \ref{sec:dom_gen_func} we introduce twisted Fourier transforms, show that they are particularly amenable to verifying their belonging to the domain of generating functionals, and use this to show that the domain of every generating functional contains a dense $*$-subalgebra. Here also we discuss the relevant domain in various concrete examples. In Section \ref{sec:cocycles} we indicate the consequences of the earlier results for the existence of quantum group cocycles. In Section \ref{sec:reconstruction} we prove two versions of the reconstruction result for conditionally positive symmetric functionals defined on a domain satisfying certain technical requirements. Finally, Section \ref{sec:cqg_examples} discusses the consequences of the main results for the case of compact quantum groups.

\section{\label{sec:prelims}Preliminaries}

We start with some conventions. Inner products are linear in the right
variable, and all inner product spaces are complex. For a Hilbert
space $\H$ and $\z\in\H$, denote by $\omega_{\z}$ the element of
$B(\H){}_{*}$ given by $T\mapsto\left\langle \z,T\z\right\rangle $,
$T\in B(\H)$. For a matrix $\left(a_{ij}\right)_{1\le i,j\le n}$,
the element in the $i$th row and $j$th column is $a_{ij}$. For
a C$^{*}$-algebra $B$, denote by $B^{\#}$ its trivial unitisation,
which is $B$ itself if the latter is unital, and by $\M B$ the multiplier
algebra of $B$. If $B$ is unital we denote its unit by $\one$.
For $\om\in B^{*}$, we use the same notation $\om$ for the strict extension of $\om$ to $\M B$ (or merely to $B^{\#}$), and do the same for slice maps.
We also let $\overline{\om}\in B^{*}$ be given
by $\overline{\om}(x):=\overline{\om(x^{*})}$, $x\in B$. We denote
by $\tensormin$ and $\tensorn$ the minimal C$^{*}$-algebraic and
normal von Neumann algebraic tensor products, respectively.

Let $\vNa$ be a von Neumann algebra acting standardly on a Hilbert space
$\Ltwo \vNa$ and $\varphi$ be a normal semi-finite faithful (n.s.f.)~weight on $\vNa$. A (non-negative) closed densely-defined quadratic
form $Q$ on $\Ltwo \vNa$ is called a \emph{Dirichlet form with respect
	to $\varphi$} if $Q\circ\pi\le Q$, where $\pi$ is the nearest-point
projection of $\Ltwo \vNa$ onto the key closed convex set associated with $(\vNa,\varphi)$ as defined in \citep[p.~62, with terminology from pp.~42 and 53]{Goldstein_Lindsay__Markov_sgs_KMS_symm_weight}.
More generally, for $n\in\N$, we write $\pi^{(n)}$ for the nearest-point projection of $\Ltwo{\matrices_{n},\Tr_{n}}\tensor\Ltwo{\vNa}$
onto the key closed convex set associated with $(\matrices_{n}\tensor \vNa,\Tr_{n}\tensor\varphi)$, where $\Tr_{n}$ is the canonical (non-normalised) trace on $\matrices_n$.
If all matrix amplifications of $Q$ are Dirichlet, namely $Q^{(n)} \circ \pi^{(n)} \le Q^{(n)}$ for all $n\in\N$, we say that $Q$ is \emph{completely}
Dirichlet with respect to $\varphi$ \citep[Appendix]{Skalski_Viselter__convolution_semigroups}.
All the related terminology can be found in \citep{Skalski_Viselter__convolution_semigroups}.

The basic objects of this paper are locally compact quantum groups
in the sense of Kustermans and Vaes. The following definition and
properties are taken from \citep{Kustermans_Vaes__LCQG_C_star,Kustermans_Vaes__LCQG_von_Neumann,Van_Daele__LCQGs}
unless otherwise indicated.
\begin{defn}
	A \emph{locally compact quantum group} in the von Neumann algebraic
	setting is a pair $\G=\left(\vNa,\Delta\right)$ that satisfies:
	\begin{enumerate}
		\item $\vNa$ is a von Neumann algebra;
		\item $\Delta\colon \vNa\to \vNa\tensorn \vNa$ is a \emph{co-multiplication} on $\vNa$,
		i.e.\ a normal unital $*$-homomorphism that is co-associative: $(\Delta\tensor\i)\circ\Delta=(\i\tensor\Delta)\circ\Delta$;
		\item there exist n.s.f.~weights $\varphi,\psi$ on $\vNa$, called the left
		and right\emph{ Haar weights}, which are left and right invariant
		under $\Delta$, respectively.
	\end{enumerate}
	Henceforth we write $\Linfty{\G}$ for $\vNa$, $\Lone{\G}$ for the
	predual $\Linfty{\G}_{*}$, and $\Ltwo{\G}$ for a Hilbert space on
	which $\Linfty{\G}$ acts standardly.
\end{defn}

For example, each locally compact group $G$ induces a locally compact
quantum group with $\vNa=\Linfty G$ and $(\Delta(f))(s,t):=f(st)$,
where we identified $\Linfty G\tensorn\Linfty G\cong\Linfty{G\times G}$.

Every locally compact quantum group $\G$ admits a \emph{dual} locally
compact quantum group $\widehat{\G}$. This duality extends Pontryagin's
duality for locally compact abelian groups, and satisfies the `double
dual property': $\widehat{\widehat{\G}}=\G$. Objects pertaining
to $\widehat{\G}$ will be adorned with a hat.

Let $\G$ be a locally compact quantum group. There exists a unitary
$W\in\Linfty{\G}\tensorn\Linfty{\widehat{\G}}$, called the \emph{left
	regular representation} of $\G$, which implements $\Delta$ by $\Delta(x)=W^{*}(\one\tensor x)W$ (acting on $\Ltwo{\G} \tensor \Ltwo{\G}$)
for all $x\in\Linfty{\G}$. The \emph{antipode} of $\G$ is a generally
unbounded, ultraweakly closed operator $\Sant$ on $\Linfty{\G}$
such that for every $\widehat{\om}\in\Lone{\widehat{\G}}$ we have
$(\i\tensor\widehat{\om})(W)\in D(\Sant)$ and $\Sant\left((\i\tensor\widehat{\om})(W)\right)=(\i\tensor\widehat{\om})(W^{*})$.
It has a `polar decomposition' $\Sant=\Rant\circ\tau_{-i/2}$, where
$\Rant$ is the \emph{unitary antipode}, which is a $*$-anti-automorphism
of $\Linfty{\G}$, and $\tau_{-i/2}$ is the generator of the \emph{scaling
	group} $\left(\tau_{t}\right)_{t\in\R}$, which is the action of $\R$
on $\Linfty{\G}$ associated with the scaling group.

There are two other `faces' of $\G$. The first is the \emph{reduced
	C$^{*}$-algebraic} face, based on a C$^{*}$-algebra $\Cz{\G}$ that
is ultraweakly dense in $\Linfty{\G}$ and satisfies $W\in\M{\Cz{\G}\tensormin\Cz{\widehat{\G}}}$.
The second is the \emph{universal C$^{*}$-algebraic} face \citep{Kustermans__LCQG_universal},
based on a C$^{*}$-algebra $\CzU{\G}$, which has a special universality
property. In particular, it surjects canonically onto $\Cz{\G}$ and
possesses a distinguished character $\epsilon$ called the \emph{co-unit}
of $\G$. The unitary $W$ has half-universal versions $\Ww,\wW$,
where, e.g., $\wW\in\M{\Cz{\G}\tensormin\CzU{\widehat{\G}}}$ and $(\omega \tensor \i)(\wW) \in \CzU{\widehat{\G}}$ for all $\omega \in \Lone{\G}$. The
co-multiplication also has a universal version $\Delta_{\mathrm{u}}\colon\CzU{\G}\to\M{\CzU{\G}\tensormin\CzU{\G}}$,
which induces on $\CzU{\G}^{*}$ a \emph{convolution} product $\star$,
and we have natural isometric embeddings $\Lone{\G}\hookrightarrow\Cz{\G}^{*}\hookrightarrow\CzU{\G}^{*}$.
Furthermore, the maps $\Sant,\Rant,\tau$ have universal versions
$\Sant^{\mathrm{u}},\Rant^{\mathrm{u}},\tau^{\mathrm{u}}$ acting
on $\CzU{\G}$.

We say that $\G$ is \emph{compact} if $\Cz \G$, equivalently $\CzU \G$, is unital \cite{Woronowicz__symetries_quantiques,Runde__charac_compact_discr_QG}. In this case, we write $((u^\a_{ij})_{1\le i,j \le n_\a})_{\a \in \Irred \G}$ for a complete family of representatives of equivalence classes of (finite-dimensional) irreducible representations of $\G$. Then $\Pol\G := \linspan\{u^\a_{ij} : \a \in \Irred \G, 1\le i,j \le n_\a\}$ is a dense subspace of $\CU \G := \CzU \G$.

\begin{defn}[\citep{Schurmann__pos_and_cond_pos_coalg}]
	\label{def:cond_pos}Let $B$ be a C$^{*}$-algebra and $\epsilon$
	be a character of $B$. A linear functional $\gamma\colon\mathscr{A}\to\C$,
	where $\mathscr{A}$ is a subspace of $B$, is called \emph{conditionally
		positive with respect to $\epsilon$} if $\gamma(a)\ge0$ for every
	$a\in\mathscr{A}\cap\ker\epsilon\cap B_{+}$. 
\end{defn}

It is obvious that $\mu+s\epsilon$ is conditionally positive with
respect to $\epsilon$ for every $\mu\in B_{+}^{*}$ and $s\in\C$;
see \prettyref{rem:bounded_cond_pos_funcl} for the converse.

\begin{defn}[\citep{Lindsay_Skalski__conv_semigrp_states}]
A \emph{convolution semigroup of positive functionals on $\CzU{\G}$}
(or \emph{on $\G$}) is a family $\left(\mu_{t}\right)_{t\ge0}$ in
$\CzU{\G}_{+}^{*}$ such that $\mu_{0}=\epsilon$ and $\mu_{s}\conv\mu_{t}=\mu_{s+t}$
for all $s,t\ge0$. Say that $\left(\mu_{t}\right)_{t\ge0}$ is \emph{$w^{*}$-continuous}
if $\mu_{t}\xrightarrow[t\to0^{+}]{}\epsilon$ in the $w^{*}$-topology.
In this case, the \emph{generating functional} of $\left(\mu_{t}\right)_{t\ge0}$
is the (generally unbounded) linear functional $\gamma$ over $\CzU{\G}$
defined by 
\[
\gamma(x):=\lim_{t\to0^{+}}\frac{\mu_{t}(x)-\epsilon(x)}{t}
\]
with maximal domain $D(\gamma)$, consisting of all $x\in\CzU{\G}$
for which this limit exists. 
\end{defn}

The generating functional $\gamma$ of a $w^{*}$-continuous convolution semigroup
of positive functionals on $\CzU{\G}$ is clearly conditionally positive with respect to the co-unit. Furthermore, if the convolution semigroup consists of states and we extend $\gamma$ to $\linspan(D(\gamma)\cup\{\one\})\subseteq\CzU{\G}^{\#}$ by making
it vanish at $\one$ (which is automatic if $\G$ is compact), then the extended functional is also conditionally positive with respect to the co-unit.

To every $\mu\in\CzU{\G}^{*}$ we associated in \citep[Subsection 2.1 and Lemma 2.14]{Skalski_Viselter__convolution_semigroups}
the operators $R_{\mu} \in CB(\Linfty \G)$ and $\widetilde{R}_{\mu}^{(2,\varphi)}\in\M{\Cz{\widehat{\G}}}$. 
Recall that the first of them is defined as the adjoint of the operator on $\Lone \G$ given by the formula $\omega \mapsto \mu \star \omega$ (as $\Lone{\G}$,
when viewed canonically as a subspace of the completely contractive
Banach algebra $\CzU{\G}^{*}$, is an ideal); and the second is its natural KMS-implementation on $\Ltwo{\G}$ with respect to $\varphi$.

The maps $\mu \mapsto R_{\mu}, \mu \mapsto \widetilde{R}_{\mu}^{(2,\varphi)}$ are linear and injective, with $R_{\epsilon} = \i, \widetilde{R}_{\epsilon}^{(2,\varphi)} = \one$.
More information is provided in \prettyref{thm:SV_thm_3_4} and \prettyref{prop:SV_prop_3_8} below.

We now quote one of the main results of \citep{Skalski_Viselter__convolution_semigroups}.
Only the relevant parts are stated; for the rest, see \citep{Skalski_Viselter__convolution_semigroups}.
We take the opportunity to fix a mistake in the statement of \citep[Theorem 0.1]{Skalski_Viselter__convolution_semigroups}:
the words `modulo multiplication of forms by a positive number'
should have been `modulo subtracting a positive multiple of the quadratic
form $\left\Vert \cdot\right\Vert ^{2}$', see \prettyref{rem:SV_thm_3_4}. 
\begin{thm}[{\citep[Theorem 3.4]{Skalski_Viselter__convolution_semigroups}}]
	\label{thm:SV_thm_3_4}Let $\G$ be a locally compact quantum group.
	There exist $1-1$ correspondences between the following classes:
	\begin{enumerate}
		\item \label{enu:SV_thm_3_4__1}$w^{*}$-continuous convolution semigroups
		$\left(\mu_{t}\right)_{t\ge0}$ of $\Rant^{\mathrm{u}}$-invariant
		contractive positive functionals on $\CzU{\G}$;
		\item \label{enu:SV_thm_3_4__2}$C_{0}$-semigroups $\left(S_{t}\right)_{t\ge0}$
		of selfadjoint completely Markov operators on $\Ltwo{\G}$ with respect
		to $\varphi$ that belong to $\Linfty{\widehat{\G}}$;
		\item \label{enu:SV_thm_3_4__3}completely Dirichlet forms $Q$ with respect
		to $\varphi$ that are invariant under $\mathcal{U}(\Linfty{\widehat{\G}}')$.
	\end{enumerate}
	The correspondences are given by $S_{t}=\widetilde{R}_{\mu_{t}}^{(2,\varphi)}$
	for all $t\ge0$ (\prettyref{enu:SV_thm_3_4__1}$\Leftrightarrow$\prettyref{enu:SV_thm_3_4__2})
	and the general correspondence between selfadjoint completely Markov
	semigroups and completely Dirichlet forms \citep[Corollary A.8]{Skalski_Viselter__convolution_semigroups}
	(\prettyref{enu:SV_thm_3_4__2}$\Leftrightarrow$\prettyref{enu:SV_thm_3_4__3}); the latter means that $\left(S_{t}\right)_{t\ge0} = {(e^{-tA})}_{t\ge0}$, where $A$ is the positive selfadjoint operator on $\Ltwo{\G}$ such that $Q = \Vert A^{1/2} \cdot \Vert^2$.
\end{thm}

\begin{rem}
	\label{rem:SV_thm_3_4}A $w^{*}$-continuous convolution semigroup
	of contractive positive functionals can be normalised to form one
	consisting of \emph{states} \citep[Remark 3.3]{Skalski_Viselter__convolution_semigroups}.
	To restate the $1-1$ correspondence of \prettyref{thm:SV_thm_3_4}
	\prettyref{enu:SV_thm_3_4__1}$\Leftrightarrow$\prettyref{enu:SV_thm_3_4__3}
	for states one has to `normalise' the completely Dirichlet form
	as well; the following text should therefore be added to the statement
	of \prettyref{enu:SV_thm_3_4__3}: \emph{modulo subtracting from $Q$
	a positive multiple of the quadratic form $\left\Vert \cdot\right\Vert ^{2}$}.
\end{rem}

\section{\label{sec:dom_gen_func}The domains of generating functionals}

In Sections \ref{sec:dom_gen_func}\textendash \ref{sec:reconstruction}
we let $\G$ be a locally compact quantum group. In this section we show that given a convolution semigroup of states, for certain (twisted) Fourier transforms it is particularly easy to determine whether they belong to the domain of the generating functional (see \prettyref{prop:gamma_domain} \prettyref{enu:gamma_domain__1}). This fact is used to establish the containment of a dense $*$-subalgebra in the domain of any generating functional (\prettyref{thm:gen_func_domain}). Several examples of the form of this algebra are then described.

We begin by defining our (twisted) Fourier transforms.

\begin{defn}
For $\widehat{\om}\in\Lone{\widehat{\G}}$ denote $\aa{\widehat{\om}}:=\tau_{i/4}^{\mathrm{u}}((\widehat{\om}\tensor\i)(\widehat{\wW})) \in \CzU{\G}$, where we remind the reader that ${(\tau_t^\mathrm{u})}_{t\in\R}$ is the universal version of the scaling group, which is an action of $\R$ on $\CzU{\G}$.
\end{defn}

One shows just as in \citep[proof of Lemma 2.14]{Skalski_Viselter__convolution_semigroups}
that $\aa{\widehat{\om}}$ is well defined
and satisfies $\left\Vert \aa{\widehat{\om}}\right\Vert \le\left\Vert \widehat{\om}\right\Vert $
for every $\widehat{\om}\in\Lone{\widehat{\G}}$. For the convenience of the reader, we repeat the short proof. 
By the properties of the antipode we have \[(\i \tensor \widehat{\om})(\Ww) \in D(\Sant^\mathrm{u}) \text{ and } \Sant^\mathrm{u}((\i \tensor \widehat{\om})(\Ww)) = (\i \tensor \widehat{\om})(\Ww^*).\]
That is, $(\i \tensor \widehat{\om})(\Ww) \in D(\tau^\mathrm{u}_{-i/2})$ and $\tau^\mathrm{u}_{-i/2}((\i \tensor \widehat{\om})(\Ww)) = \Rant^\mathrm{u}((\i \tensor \widehat{\om})(\Ww^*))$. Taking adjoints, we obtain $(\overline{\widehat{\om}} \tensor \i)(\widehat{\wW}) \in D(\tau^\mathrm{u}_{i/2})$ and $\tau^\mathrm{u}_{i/2}((\overline{\widehat{\om}} \tensor \i)(\widehat{\wW})) = \Rant^\mathrm{u}((\i \tensor \overline{\widehat{\om}})(\Ww))$. Thus, $(\overline{\widehat{\om}} \tensor \i)(\widehat{\wW}) \in D(\tau^\mathrm{u}_{i/4})$, and furthermore, since $\| (\overline{\widehat{\om}} \tensor \i)(\widehat{\wW}) \|, \| \Rant^\mathrm{u}((\i \tensor \overline{\widehat{\om}})(\Ww)) \| \le \| \overline{\widehat{\om}} \|$, we infer from the Phragmen\textendash Lindel\"{o}f three lines theorem that $ \| \tau^\mathrm{u}_{i/4}((\overline{\widehat{\om}} \tensor \i)(\widehat{\wW})) \| \le \| \overline{\widehat{\om}} \|$, as desired.

The map $\widehat{\om}\mapsto\aa{\widehat{\om}}$
is an injective homomorphism from $\Lone{\widehat{\G}}$ to $\CzU{\G}$
because the maps $\tau_{i/4}^{\mathrm{u}}\colon D(\tau^\mathrm{u}_{i/4})\to\CzU{\G}$ and $\Lone{\widehat{\G}}\ni\widehat{\om}\mapsto(\widehat{\om}\tensor\i)(\widehat{\wW})$
are injective homomorphisms. We require the following result from
\citep{Skalski_Viselter__convolution_semigroups}.
\begin{prop}[{\citep[Proposition 3.8 and its proof]{Skalski_Viselter__convolution_semigroups}}]
\label{prop:SV_prop_3_8}\mbox{}
\begin{enumerate}
\item \label{enu:SV_prop_3_8__1}For every $\nu\in\CzU{\G}^{*}$ and $\widehat{\om}\in\Lone{\widehat{\G}}$
we have
\[
\widehat{\om}(\widetilde{R}_{\nu}^{(2,\varphi)})=\nu(\aa{\widehat{\om}}).
\]
\item \label{enu:SV_prop_3_8__2}Consider a $w^{*}$-continuous convolution
semigroup $\left(\mu_{t}\right)_{t\ge0}$ of $\Rant^{\mathrm{u}}$-invariant
contractive positive functionals on $\CzU \G$. Denote by $Q$ the completely Dirichlet form associated
to $\left(\mu_{t}\right)_{t\ge0}$, and let $\gamma$ be the generating
functional of $\left(\mu_{t}\right)_{t\ge0}$. Then $D(Q)=\{ \z\in\Ltwo{\G} : \aa{\widehat{\om}_{\z}}\in D(\gamma)\} $
and for every $\z\in D(Q)$ we have $Q\z=-\gamma(\aa{\widehat{\om}_{\z}})$.
\end{enumerate}
\end{prop}

The twisted Fourier transforms interact in a natural manner with the unitary antipode.

\begin{lem}
\label{lem:D_plus}\mbox{}
\begin{enumerate}
\item \label{enu:D_plus__1}For every $\widehat{\om}\in\Lone{\widehat{\G}}$
we have $\aa{\widehat{\om}}^{*}=\aa{\overline{\widehat{\om}}\circ\widehat{\Rant}}$
and $\Rant^{\mathrm{u}}(\aa{\widehat{\om}})=\aa{\widehat{\om}\circ\widehat{\Rant}}$.
\item \label{enu:D_plus__2}For every $\widehat{\om}\in\Lone{\widehat{\G}}_{+}$
we have $\left\Vert \aa{\widehat{\om}}\right\Vert = \left\Vert \widehat{\om}\right\Vert =\epsilon(\aa{\widehat{\om}})$.
\item \label{enu:D_plus__3}For every $\widehat{\om}\in\Lone{\widehat{\G}}_{+}$
we have $\left[(\i+\Rant^{\mathrm{u}})(\epsilon(\cdot)\one-\i)\right](\aa{\widehat{\om}})=2\left(\left\Vert \aa{\widehat{\om}}\right\Vert\one-\Ree(\aa{\widehat{\om}})\right)\ge0$
in $\CzU{\G}^{\#}$.
\item \label{enu:D_plus__4}The set 
\[
\mathscr{D}_{+}:=\bigl\{\aa{\widehat{\om}}:\widehat{\om}\in\Lone{\widehat{\G}}_{+}\bigr\}\subseteq\CzU{\G}
\]
 is a cone, and it is selfadjoint, globally $\Rant^{\mathrm{u}}$-invariant,
and closed under multiplication.
\end{enumerate}
\end{lem}

\begin{proof}
\prettyref{enu:D_plus__1} Let $\widehat{\om}\in\Lone{\widehat{\G}}$.
The second identity follows readily since $\Rant^{\mathrm{u}}$ commutes
with $\tau^{\mathrm{u}}$ and $(\widehat{\Rant}\tensor\Rant^{\mathrm{u}})(\widehat{\wW})=\widehat{\wW}$.
Similarly, since $\Sant^{\mathrm{u}}=\Rant^{\mathrm{u}}\circ\tau_{-i/2}^{\mathrm{u}}$,
we have
\[
\begin{split}\aa{\widehat{\om}}^{*}=\tau_{i/4}^{\mathrm{u}}\bigl((\widehat{\om}\tensor\i)(\widehat{\wW})\bigr)^{*} & =\tau_{-i/4}^{\mathrm{u}}\bigl((\overline{\widehat{\om}}\tensor\i)(\widehat{\wW}^{*})\bigr)=\tau_{-i/4}^{\mathrm{u}}\bigl((\i\tensor\overline{\widehat{\om}})(\Ww)\bigr)\\
 & =\tau_{-i/4}^{\mathrm{u}}\bigl[(\Sant^{\mathrm{u}})^{-1}\bigl((\i\tensor\overline{\widehat{\om}})(\Ww^{*})\bigr)\bigr]=\tau_{i/4}^{\mathrm{u}}\bigl[\Rant^{\mathrm{u}}\bigl((\i\tensor\overline{\widehat{\om}})(\Ww^{*})\bigr)\bigr]\\
 & =\tau_{i/4}^{\mathrm{u}}\bigl((\i\tensor(\overline{\widehat{\om}}\circ\widehat{\Rant}))(\Ww^{*})\bigr)=\tau_{i/4}^{\mathrm{u}}\bigl(((\overline{\widehat{\om}}\circ\widehat{\Rant})\tensor\i)(\widehat{\wW})\bigr)=\aa{\overline{\widehat{\om}}\circ\widehat{\Rant}}.
\end{split}
\]

\prettyref{enu:D_plus__2} For every $\widehat{\om}\in\Lone{\widehat{\G}}_{+}$
we have $\left\Vert \aa{\widehat{\om}}\right\Vert \le\left\Vert \widehat{\om}\right\Vert =\widehat{\om}(\one)=\epsilon(\aa{\widehat{\om}}) \le \left\Vert \aa{\widehat{\om}}\right\Vert$.

\prettyref{enu:D_plus__3} Combine \prettyref{enu:D_plus__1} with \prettyref{enu:D_plus__2}.

\prettyref{enu:D_plus__4} Since $\Lone{\widehat{\G}}_{+}$ is a cone
in $\Lone{\widehat{\G}}$ that is closed under convolution and the
map $\Lone{\widehat{\G}}\ni\widehat{\om}\mapsto\aa{\widehat{\om}}$
is a homomorphism, $\mathscr{D}_{+}$ is a cone that is closed under
multiplication. Selfadjointness and global $\Rant^{\mathrm{u}}$-invariance
of $\mathscr{D}_{+}$ follow from \prettyref{enu:D_plus__1}.
\end{proof}

Although generating functionals are in general not closed in any natural topology, so that the notion of a core of these objects does not make sense as such, the cone $\mathscr{D}_{+}$ introduced above possesses certain core-like properties, as the next proposition shows.

\begin{prop}
\label{prop:gen_func_conv_semi_detrm_by_D_plus}
The generating functional $\gamma$ of a $w^{*}$-continuous
convolution semigroup of $\Rant^{\mathrm{u}}$-invariant contractive positive functionals on
$\CzU \G$ is uniquely determined by its behaviour on $\mathscr{D}_{+}$ (namely,
by $\mathscr{D}_{+}\cap D(\gamma)$ and by $\gamma|_{\mathscr{D}_{+}\cap D(\gamma)}$).
\end{prop}
\begin{proof}
 \prettyref{prop:SV_prop_3_8}~\prettyref{enu:SV_prop_3_8__2} implies that the behaviour of $\gamma$ on $\mathscr{D}_{+}$ determines uniquely the completely Dirichlet form associated to the respective convolution semigroup, and hence, by \prettyref{thm:SV_thm_3_4} correspondence \prettyref{enu:SV_thm_3_4__1}$\Leftrightarrow$\prettyref{enu:SV_thm_3_4__3} also the convolution semigroup itself. Thus it also determines its generator, $\gamma$.
\end{proof}

Denote by $\overline{\mathscr{D}_{+}}$ the norm closure of $\mathscr{D}_{+}$
in $\CzU{\G}$, which equals its $w$-closure in $\CzU{\G}$ by \prettyref{lem:D_plus}~\prettyref{enu:D_plus__4}
and the Hahn\textendash Banach theorem, as it is convex by Lemma \ref{lem:D_plus} \ref{enu:D_plus__4}.
\begin{prop}
\label{prop:gamma_domain}Let $\left(\mu_{t}\right)_{t\ge0}$ be a
$w^{*}$-continuous convolution semigroup of $\Rant^{\mathrm{u}}$-invariant
contractive positive functionals on $\CzU{\G}$ and $\gamma$ be its generating functional. 
\begin{enumerate}
\item \label{enu:gamma_domain__1}For every $a\in\overline{\mathscr{D}_{+}}$,
the function $\left(0,\infty\right)\ni t\mapsto\frac{1}{t}(\epsilon-\mu_{t})(a)$
is non-negative and decreasing, and $a\in D(\gamma)$ if and only
if $\left\{ \frac{1}{t}(\epsilon-\mu_{t})(a): t>0\right\} $ is bounded.
\item \label{enu:gamma_domain__2}The set $\mathscr{D}_{+}\cap D(\gamma)$
is total in $\CzU{\G}$.
\item \label{enu:gamma_domain__3}The following lower semi-continuity property
of $\gamma$ holds: if $\left(a_{i}\right)_{i\in\mathcal{I}}$ is
a net in $\overline{\mathscr{D}_{+}}\cap D(\gamma)$ converging in
the $w$-topology of $\CzU{\G}$ to some $a\in\CzU{\G}$ and $\liminf_{i\in\mathcal{I}}(-\gamma(a_{i}))<\infty$,
then $a\in D(\gamma)$ and $0\le-\gamma(a)\le\liminf_{i\in\mathcal{I}}(-\gamma(a_{i}))$.
\end{enumerate}
\end{prop}

\begin{proof}
\prettyref{enu:gamma_domain__1} Let $A$ be the (generally unbounded)
positive selfadjoint operator on $\Ltwo{\G}$ such that $\widetilde{R}_{\mu_{t}}^{(2,\varphi)}=e^{-tA}$
for every $t\ge0$ (see \prettyref{thm:SV_thm_3_4}).
Let first $a\in\mathscr{D}_{+}$ and write $a=\aa{\widehat{\om}}$,
$\widehat{\om}\in\Lone{\widehat{\G}}_{+}$. For every $t>0$ we have
$\frac{1}{t}(\epsilon-\mu_{t})(a)=\widehat{\om}(\widetilde{R}_{\frac{1}{t}(\epsilon-\mu_{t})}^{(2,\varphi)})=\widehat{\om}(\frac{1}{t}(\one-e^{-tA}))$
by \prettyref{prop:SV_prop_3_8}~\prettyref{enu:SV_prop_3_8__1}.
The first assertion, namely that $\left(0,\infty\right)\ni t\mapsto\frac{1}{t}(\epsilon-\mu_{t})(a)$
is non-negative and decreasing, thus follows from functional calculus
as the function $\left(0,\infty\right)\ni t\mapsto\frac{1}{t}(1-e^{-t})$
is non-negative and decreasing. Consequently, this first assertion is readily seen
to hold for $a\in\overline{\mathscr{D}_{+}}$. The second assertion
is now immediate.

\prettyref{enu:gamma_domain__2} Suppose that $\nu\in\CzU{\G}^{*}$
satisfies $\nu(\mathscr{D}_{+}\cap D(\gamma))=\left\{ 0\right\} $.
Then by \prettyref{prop:SV_prop_3_8} we have $\widehat{\om}_{\z}(\widetilde{R}_{\nu}^{(2,\varphi)})=0$
for every $\z$ in the dense subspace $D(Q)$ of $\Ltwo{\G}$,
so that $\widetilde{R}_{\nu}^{(2,\varphi)}=0$. This is equivalent to $R_{\nu}=0$,
hence $\nu=0$ by \citep[Theorem 2.1 (a)]{Skalski_Viselter__convolution_semigroups} (alternatively, use \citep[Lemma 2.17, (c)$\implies$(a)]{Skalski_Viselter__convolution_semigroups}).

\prettyref{enu:gamma_domain__3} Let $\left(a_{i}\right)_{i\in\mathcal{I}}$
be a net in $\overline{\mathscr{D}_{+}}\cap D(\gamma)$ that converges
to $a\in\CzU{\G}$ in the $w$-topology of $\CzU{\G}$ (so that $a\in\overline{\mathscr{D}_{+}}$).
By \prettyref{enu:gamma_domain__1} we have 
\[
0\le\frac{1}{t}(\epsilon-\mu_{t})(a_{i})\le-\gamma(a_{i})\qquad(\forall_{0<t}\forall_{i\in\mathcal{I}}).
\]
Taking the limit (inferior) as $i\in\mathcal{I}$ we get $0\le\frac{1}{t}(\epsilon-\mu_{t})(a)\le\liminf_{i\in\mathcal{I}}(-\gamma(a_{i}))$
for $0<t$. Applying \prettyref{enu:gamma_domain__1} again we deduce
that if $\liminf_{i\in\mathcal{I}}(-\gamma(a_{i}))<\infty$ then $a\in D(\gamma)$
and $0\le-\gamma(a)\le\liminf_{i\in\mathcal{I}}(-\gamma(a_{i}))$.
\end{proof}
\begin{cor}
\label{cor:gamma_domain_domination}Let $\left(\mu_{t}\right)_{t\ge0}$
be a $w^{*}$-continuous convolution semigroup of $\Rant^{\mathrm{u}}$-invariant
contractive positive functionals of $\CzU{\G}$ and $\gamma$ be its generating functional.
Also let $\widehat{\om}_{1},\widehat{\om}_{2}\in\Lone{\widehat{\G}}_{+}$
and assume that $\aa{\widehat{\om}_{2}}\in D(\gamma)$.
\begin{enumerate}
\item \label{enu:gamma_domain_domination__1}If $\widehat{\om}_{1}\le\widehat{\om}_{2}$
then $\aa{\widehat{\om}_{1}}\in D(\gamma)$.
\item \label{enu:gamma_domain_domination__2}If $\left(\mu_{t}\right)_{t\ge0}$ consists of states and $\left[(\i+\Rant^{\mathrm{u}})(\epsilon(\cdot)\one-\i)\right](\aa{\widehat{\om}_{1}})\le\left[(\i+\Rant^{\mathrm{u}})(\epsilon(\cdot)\one-\i)\right](\aa{\widehat{\om}_{2}})$
in $\CzU{\G}^{\#}$ then $\aa{\widehat{\om}_{1}}\in D(\gamma)$.
\end{enumerate}
\end{cor}

\begin{proof}
In both cases we will show that 
\begin{equation}
(\epsilon-\mu_{t})(\aa{\widehat{\om}_{1}})\le(\epsilon-\mu_{t})(\aa{\widehat{\om}_{2}})\qquad(\forall_{t>0}),\label{eq:gamma_domain_domination}
\end{equation}
hence $\aa{\widehat{\om}_{1}}\in D(\gamma)$ by \prettyref{prop:gamma_domain}~\prettyref{enu:gamma_domain__1}.

\prettyref{enu:gamma_domain_domination__1} By assumption, $\aa{\widehat{\om}_{2}}-\aa{\widehat{\om}_{1}}=\aa{\widehat{\om}_{2}-\widehat{\om}_{1}}\in\mathscr{D}_{+}$.
\prettyref{prop:gamma_domain}~\prettyref{enu:gamma_domain__1} thus
implies \prettyref{eq:gamma_domain_domination}.

\prettyref{enu:gamma_domain_domination__2} For all $\widehat{\om}\in\Lone{\widehat{\G}}$
and $t>0$ we have $\mu_{t}\left\{ \left[(\i+\Rant^{\mathrm{u}})(\epsilon(\cdot)\one-\i)\right](\aa{\widehat{\om}})\right\} =2(\epsilon-\mu_{t})(\aa{\widehat{\om}})$.
Thus, the assumed inequality entails \prettyref{eq:gamma_domain_domination}.
\end{proof}

The following lemma is elementary.
\begin{lem}
\label{lem:mu_a_1_a_2}Let $B$ be a C$^{*}$-algebra, $\mu\in B^{*}$
be a state and $a_{1},a_{2}\in B$ be contractions. Write $c_{i}:=1-\Ree\mu(a_{i})$,
$i=1,2$. Then $1-\Ree\mu(a_{1}a_{2})\le c_{1}+c_{2}+2\sqrt{c_{1}c_{2}}$.
\end{lem}

\begin{proof}
Let $(\H,\xi)$ be the GNS construction for $(B,\mu)$ (suppressing
the representation). Then 
\begin{equation}
\begin{split}1-\Ree\mu(a_{1}a_{2}) & =\Ree\left\langle \xi,\xi-a_{1}a_{2}\xi\right\rangle =\Ree\left\langle \xi,\xi-a_{1}\xi\right\rangle +\Ree\left\langle \xi,a_{1}(\xi-a_{2}\xi)\right\rangle \\
 & =c_{1}+\Ree\left\langle a_{1}^{*}\xi,\xi-a_{2}\xi\right\rangle =c_{1}+c_{2}+\Ree\left\langle a_{1}^{*}\xi-\xi,\xi-a_{2}\xi\right\rangle .
\end{split}
\label{eq:mu_a_1_a_2}
\end{equation}
Since $a_{1},a_{2}$ are contractions, we have
\[\|a_1^* \xi- \xi\|^2 = \|a_1^* \xi\|^2 + \|\xi\|^2 - 2 \Ree \langle \xi, a_1^* \xi \rangle \leq 2 - 2 \Ree  \langle \xi, a_1^* \xi \rangle = 2 c_1\]
and, similarly, $\|\xi- a_2\xi\|^2  \leq 2 c_2$.
Thus, \prettyref{eq:mu_a_1_a_2} and
the Cauchy\textendash Schwarz inequality imply the desired inequality.
\end{proof}

The next result is the main theorem of this section.

\begin{thm}
\label{thm:gen_func_domain}Let $\G$ be a locally compact quantum
group, $\left(\mu_{t}\right)_{t\ge0}$ be a $w^{*}$-continuous convolution
semigroup of $\Rant^{\mathrm{u}}$-invariant states of $\CzU{\G}$
and $\gamma$ be its generating functional. Then $\linspan(\mathscr{D}_{+}\cap D(\gamma))$
is a globally $\Rant^{\mathrm{u}}$-invariant dense $*$-subalgebra
of $\CzU{\G}$, and so is $\linspan(\overline{\mathscr{D}_{+}}\cap D(\gamma))$.
\end{thm}

\begin{proof}
By \prettyref{lem:D_plus}~\prettyref{enu:D_plus__4} and \prettyref{prop:gamma_domain}~\prettyref{enu:gamma_domain__2}
it suffices to show that $\{a_{1}a_{2}: a_{1},a_{2}\in\overline{\mathscr{D}_{+}}\cap D(\gamma)\}\subseteq D(\gamma)$. 

We shall require two properties of the elements of $\overline{\mathscr{D}_{+}}$.
First, by \prettyref{lem:D_plus}~\prettyref{enu:D_plus__2}, $\left\Vert a\right\Vert =\epsilon(a)$ for all $a\in\mathscr{D}_{+}$
and thus for all $a\in\overline{\mathscr{D}_{+}}$.
Second, by \prettyref{prop:SV_prop_3_8}~\prettyref{enu:SV_prop_3_8__1} and positivity of $\widetilde{R}_{\mu_{t}}^{(2,\varphi)}$
we have $\mu_{t}(\mathscr{D}_{+})\subseteq[0,\infty)$, and thus $\mu_{t}(\overline{\mathscr{D}_{+}})\subseteq[0,\infty)$,
for every $t\ge0$.

Let $a_{1},a_{2}\in\overline{\mathscr{D}_{+}}\cap D(\gamma)$ be of
norm $1$. The assumption that $a_{i}\in D(\gamma)$ gives $0\le C<\infty$
such that $1-\mu_{t}(a_{i})=(\epsilon-\mu_{t})(a_{i})\le Ct$ for
all $t\ge0$ ($i=1,2$). From \prettyref{lem:mu_a_1_a_2} we infer
that $\frac{1}{t}(\epsilon-\mu_{t})(a_{1}a_{2})=\frac{1}{t}(1-\mu_{t}(a_{1}a_{2}))\le4C$
for all $t>0$, so \prettyref{prop:gamma_domain}~\prettyref{enu:gamma_domain__1}
implies that $a_{1}a_{2}\in D(\gamma)$.
\end{proof}

\begin{rem}
A $w^{*}$-continuous convolution semigroup of $\Rant^{\mathrm{u}}$-invariant states $\left(\mu_{t}\right)_{t\ge0}$ determines a $C_0$-semigroup of completely positive contractions $\left(T^\mathrm{u}_{t}\right)_{t\ge0}$ on $\CzU{\G}$, given simply by $T_t^\mathrm{u} = (\textup{id} \otimes \mu_t)\circ \Delta_\mathrm{u}$ ($t \geq 0$). This semigroup admits a densely-defined generator $L\colon D(L)\to \CzU{\G}$ and it is not difficult to see that $D(L) \subseteq\{a\in\CzU{\G}: \forall_{\nu \in \CzU{\G}^*} \, (\nu \otimes \textup{id}) (\Delta_\mathrm{u}(a)) \in D(\gamma) \}$, where $\gamma$ is the generating functional of  $\left(\mu_{t}\right)_{t\ge0}$ \cite[Proposition 3.6]{Lindsay_Skalski__conv_semigrp_states}. 
\prettyref{prop:SV_prop_3_8}~\prettyref{enu:SV_prop_3_8__1} implies that for every $\widehat\omega \in \Lone{\widehat\G}$ and $t \ge 0$ we have $T^\mathrm{u}_{t} \aa{\widehat\omega} = \aa{\widehat\omega \cdot \widetilde{R}_{\mu_t}^{(2,\varphi)}}$, thus $\int_0^t T^\mathrm{u}_{s} \aa{\widehat\omega} \d s = \aa{\widehat\omega \cdot \int_0^t \widetilde{R}_{\mu_s}^{(2,\varphi)} \mathrm{d} s}$ by \citep[Lemma 2.17]{Skalski_Viselter__convolution_semigroups}, where the left-side integral is in norm and the right-side integral is in the strict topology. The latter element belongs to $\linspan(\mathscr{D}_{+}) \cap D(L)$. We conclude that $\linspan(\mathscr{D}_{+}) \cap D(L)$ is a dense $*$-subspace of $\CzU{\G}$. It is also ${(T_t^\mathrm{u})}_{t\ge 0}$-invariant, thus it is a core of $L$. We do not know however if it is an algebra.
\end{rem}

\begin{rem}\label{rem:bounded_cond_pos_funcl}
For a locally compact group $G$, every \emph{bounded} conditionally
positive-definite function $\theta \colon G\to\R$ has the form $\theta=\mu+m$
for a positive-definite function $\mu$ on $G$ and $m\in\C$. This classical
statement generalises to a wider setting, as can be deduced from results of  \cite[Section 6]{Lindsay_Skalski__quant_stoch_conv_cocyc_3}: if $B$ is a unital C$^{*}$-algebra
	with a character $\epsilon$, then every (not necessarily hermitian, and a priori not necessarily bounded) conditionally positive
	linear functional $\gamma\colon B \to \mathbb{C}$ has the form $\gamma=s\mu-(s-\gamma(\one))\epsilon$
	for a state $\mu$ of $B$ and $s\ge0$.
\end{rem}

We now check how the subalgebras $\linspan(\mathscr{D}_{+}\cap D(\gamma))$
and $\linspan(\overline{\mathscr{D}_{+}}\cap D(\gamma))$ obtained
in \prettyref{thm:gen_func_domain} are related to the `natural' subalgebras of $D(\gamma)$
in several examples.
\begin{example}[{continuing \citep[Subsection 5.1]{Skalski_Viselter__convolution_semigroups}}]
\label{exa:our_algebra_dual_to_classical_case}Suppose that $\G:=\widehat{G}$
for a locally compact group $G$. Write $\lambda_{\mathrm{u}}$ for
the canonical embedding of $\Lone G$ into $\CzU{\G}=\CStarF G$.
We have $\mathscr{D}_{+}=\lambda_{\mathrm{u}}({\Lone{G}}_{+})$. Then
$w^{*}$-continuous convolution semigroups of $\Rant^{\mathrm{u}}$-invariant
states of $\CStarF G$ correspond to $w^{*}$-continuous multiplicative
semigroups of real-valued normalised positive-definite functions on
$G$, as well as to hermitian conditionally negative-definite functions $\theta\colon G\to\R$ vanishing at $e$
(a posteriori having non-negative values), where the second
correspondence is given by sending $\theta$ to ${(e^{-t\theta})}_{t\ge0}$.
We have $\lambda_{\mathrm{u}}(\Cc G)\subseteq D(\gamma)$ and $\gamma(\lambda_{\mathrm{u}}(f))=-\int_{G}f(t)\theta(t)\d \mu(t)$
for all $f\in\Cc G$. Since $\Cc G=\linspan(\Cc G_{+})$, we clearly
have $\lambda_{\mathrm{u}}(\Cc G)\subseteq\linspan(\mathscr{D}_{+}\cap D(\gamma))$. In fact an easy argument, for example using \prettyref{prop:gamma_domain}~\prettyref{enu:gamma_domain__1}, shows that 
$\mathscr{D}_{+}\cap D(\gamma) = \lambda_{\mathrm{u}}(\{f\in \Lone{G}_{+}: f\theta \in L^1(G)\})$,
thus
$\linspan(\mathscr{D}_{+}\cap D(\gamma)) = \lambda_{\mathrm{u}}(\{f\in \Lone{G}: f\theta \in L^1(G)\})$.
\end{example}

\begin{example}
\label{exa:our_algebra_classical_case}Suppose that $\G:=G$ for a
locally compact group $G$. Then $\CzU{\G}=\Cz G$. We use the standard
notation of \citep{Eymard__Fourier_alg}; so that $A(G)$ stands
for the Fourier algebra of $G$, $P(G)$ denotes the set of all (continuous) positive-definite functions on $G$, $P_\lambda(G)$ the set of these elements of $P(G)$ whose associated representations of $G$ are weakly contained in the left regular representation, and $B(G)$ denotes the Fourier--Stieltjes algebra of $G$, i.e.\ the linear span of $P(G)$. We have $\mathscr{D}_{+}=A(G)\cap P(G)$.
We claim that $\overline{\mathscr{D}_{+}}=\Cz G\cap P_{\lambda}(G)$.
Indeed, the inclusion `$\subseteq$' is not difficult. For `$\supseteq$',
recall that every $f\in P_{\lambda}(G)$ is the limit, in the topology
of uniform convergence on compact subsets of $G$, of a net $\left(f_{i}\right)_{i\in\mathcal{I}}$
in $A(G)\cap P(G)$ by \citep[Proposition 18.3.5]{Dixmier__C_star_English},
which is necessarily eventually bounded. This topology is equivalent
to the strict topology of $\Cb G=\M{\Cz G}$ on bounded subsets of
this space. If $f\in\Cz G\cap P_{\lambda}(G)$, then $\left(f_{i}\right)_{i\in\mathcal{I}}$
consequently converges to $f$ in the $w$-topology of $\Cz G$, as by
Cohen's factorisation theorem, every continuous functional on a C$^*$-algebra $A$ is of the form $a\mapsto \omega(ba)$, where $b \in A$ and $\omega \in A^*$. This proves that $f\in\overline{\mathscr{D}_{+}}$ and the claim follows.

Consider the generating functional $\gamma\colon D(\gamma)\subseteq\Cz G\to\C$
of a $w^{*}$-continuous convolution semigroup of $\Rant^{\mathrm{u}}$-invariant
states of $\Cz G$, that is, a $w^{*}$-continuous convolution semigroup
of symmetric regular Borel measures of $G$. By Hunt's theorem, when
$G$ is a Lie group we have $\CzX{2,l}G\subseteq D(\gamma)$ (see
\citep[Theorem 4.2.8]{Heyer__book_prob_meas} or \citep[Theorem 1.1]{Liao__book},
or the extended version \citep[Theorem 4.5.9]{Heyer__book_prob_meas}
for arbitrary locally compact groups). However, it is not always true
that $\CzX{2,l}G\subseteq\linspan(\overline{\mathscr{D}_{+}}\cap D(\gamma))$,
or even that $\CzX{2,l}G\subseteq B(G)$. For instance, we have $\CzX 2{\R}\nsubseteq B(\R)$,
and actually $\CzX{\infty}{\R}\nsubseteq B(\R)$, as shown by the following
example communicated to us by Przemys{\ldash}aw Ohrysko. Let $f\in\CzX{\infty}{\R}$
be such that $f\equiv0$ on $(-\infty,0]$ and $f(x)=\frac{1}{\ln x}$
for all $x\in[2,\infty)$. Assume by contradiction that $f\in B(\R)$.
By the theorem of F.~and M.~Riesz \citep[Theorem 8.2.7]{Rudin__book_Fourier_anal}
we then have $f\in A(\R)$. Consider the function $f_{\text{odd}}\in A(\R)$
given by $\R\ni x\mapsto f(x)-f(-x)$. Since it is odd, the set $\bigl\{\int_{1}^{b}\frac{f_{\text{odd}}(x)}{x}\d x: b>1\bigr\}$
is bounded by \citep[I.4.1]{Stein_Weiss__book}, contradicting the
fact that $f_{\text{odd}}(x)=\frac{1}{\ln x}$ for $x\in[2,\infty)$.

Nonetheless, a smaller, yet still canonical, subalgebra of $\CzX{2,l}G\subseteq D(\gamma)$
is $\CcX{4,l}G$. Let us show that $\CcX 4{\R}$ is contained in $\linspan(\CzX 2{\R}\cap P(\R))$,
thus in $\linspan(\overline{\mathscr{D}_{+}}\cap D(\gamma))$. Denote
by $x$ the identity function on $\R$. Recall from \citep[p.~143, Exercise 7]{Katznelson_harmonic_analysis}
that $\CcTwo{\R}\subseteq A(\R)$; this is because for $f\in\CcTwo{\R}$
the inversion formula $f=\hat{g}$ for $g=\frac{1}{2\pi}\hat{f}(-\cdot)$
holds, as $x^{2}\hat{f}$ belongs to $A(\R)$ by \citep[Chapter VI, Theorem 1.5]{Katznelson_harmonic_analysis}
and is thus bounded, so that $\hat{f}\in\Lone{\R}$. Let $f\in\CcX 4{\R}$.
Then as before, $x^{4}\hat{f}$ is bounded, so that $x^{2}\hat{f}\in\Lone{\R}$.
Write $\hat{f}$ as the linear combination of $F_{1},\ldots,F_{4}\in L^{1}(\R)_+$
in the standard way. For $1\le i\le4$ we get $x^{2}F_{i}\in\Lone{\R}$,
and consequently \citep[Chapter VI, Theorem 1.6]{Katznelson_harmonic_analysis}
implies that $\widehat{F_{i}}\in\CzX 2{\R}$. By the foregoing, $\hat{\hat{f}}$,
and thus also $f$, belong to $\linspan(\CzX 2{\R}\cap P(\R))$.
\end{example}

\begin{example}\label{exa:our_algebra_compact_case}
Suppose that $\G$ is a compact quantum group. Then $\Pol{\G}\subseteq D(\gamma)$.
Since $\tau_{i/4}^{\mathrm{u}}$ restricts to an isomorphism of $\Pol{\G}$
we have $\Pol{\G}=\linspan(\mathscr{D}_{+}\cap\Pol{\G})$, hence $\Pol{\G}\subseteq\linspan(\mathscr{D}_{+}\cap D(\gamma))$.
\end{example}

\begin{example}
	We will discuss here convolution semigroups arising from closed quantum subgroups and the special instances of the Brownian motions on $SU_q(2)$ and $E_\mu(2)$.
	For the notions of (closed) quantum subgroups we refer to  \cite{Daws_Kasprzak_Skalski_Soltan__closed_q_subgroups_LCQGs}.
Given a closed quantum subgroup $\QH$ (in the sense of Vaes) of a locally compact quantum group $\G$ and a $w^{*}$-continuous 
convolution semigroup of states ${(\mu_t^\QH)}_{t\geq 0}$ of $\CzU{\QH}$ we define the associated $w^{*}$-continuous convolution semigroup of states ${(\mu_t^\G)}_{t\geq 0}$ of $\CzU{\G}$ simply by putting $\mu_t^\G := \mu_t^\QH \circ \Theta$, where $\Theta\colon\CzU{\G} \to \CzU{\QH}$ is the quantum subgroup-defining surjection. Denote the respective generating functionals by $\gamma^\QH$ and $\gamma^\G$. Then $D(\gamma^\G)=\{a\in \CzU{\G}: \Theta(a) \in D(\gamma^\QH)\}$, with $\gamma^\G(a)=\gamma^\QH (\Theta(a))$ for $a\in D(\gamma^\G)$. By the definition of the Vaes closed quantum subgroup, see for example \cite[Theorem 3.7]{Daws_Kasprzak_Skalski_Soltan__closed_q_subgroups_LCQGs}, we have $\Theta(\{(\widehat{\om}\tensor\i)(\widehat{\wW^\G}):\widehat{\om}\in\Lone{\widehat{\G}}\} = \{(\widehat{\om}\tensor\i)(\widehat{\wW^\QH}):\widehat{\om}\in\Lone{\widehat{\QH}}\}$.  Furthermore, the morphism $\Theta$ intertwines the scaling groups: $\Theta \circ \tau_t^{\mathrm{u},\G} = \tau_t^{\mathrm{u},\QH} \circ \Theta$ for all $t\in\R$; indeed, using the terminology of \cite{Daws_Kasprzak_Skalski_Soltan__closed_q_subgroups_LCQGs}, this intertwining holds true for every strong quantum homomorphism between two locally compact quantum groups, as follows by combining \cite[Proposition 3.10]{Meyer_Roy_Woronowicz__hom_quant_grps} with  \cite[formula (1.12)]{Daws_Kasprzak_Skalski_Soltan__closed_q_subgroups_LCQGs} relating a strong quantum homomorphism to its associated bicharacter. Thus, we have
$\mathscr{D}_{+}^\QH = \Theta (\mathscr{D}_{+}^\G)$. Hence
$\mathscr{D}_{+}^\G\cap D(\gamma^\G) = \{a\in \mathscr{D}_{+}^\G: \Theta(a) \in 
 D(\gamma^\QH)\}$,
so that $\linspan(\mathscr{D}_{+}^\G\cap D(\gamma^\G)) = \linspan( \{a\in \mathscr{D}_{+}^\G: \Theta(a) \in 
D(\gamma^\QH)\})$. Similarly, $\Theta$ intertwines the unitary antipodes, so that if the elements of ${(\mu_t^\QH)}_{t\geq 0}$ are $\Rant^{\mathrm{u},\QH}$-invariant, then the elements of ${(\mu_t^\G)}_{t\geq 0}$ are $\Rant^{\mathrm{u},\G}$-invariant.

Consider then the special instance of this construction arising from the essentially unique quantum Gaussian process on $SU_q(2)$, with $q \in [-1,1]\setminus\{0\}$, as described for example in \cite{ScSk}. Our quantum group $\G$ is in this case Woronowicz's $SU_q(2)$, and its closed quantum subgroup $\QH$ will be the circle, $\mathbb{T}$. We refer for the details of the construction to \cite[Subsection 2.3 and Section 5]{FST}. The convolution semigroup on $\mathbb{T}$ we will be interested in  is the classical heat semigroup, with the generating functional given formally by the second derivative at $0$. For us it will be easier to view $\CC{\mathbb{T}}$ as $C^*(\Z)$, so that we can use the techniques introduced in Example \ref{exa:our_algebra_dual_to_classical_case}. In this picture the heat convolution semigroup ${(\mu_t^{\mathbb{T}})}_{t\geq 0}$ corresponds to the conditionally negative-definite function $\theta(n) = n^2, \;\; n\ \in \Z$. Denote the identity function in $\CC{\mathbb{T}}$ by $z$. Then by \prettyref{exa:our_algebra_dual_to_classical_case} we have 
\[
\linspan(\mathscr{D}_{+}^{\mathbb{T}}\cap D(\gamma^{\mathbb{T}})) = \Big\{\sum_{n\in \Z} a_n z^n: ({a_n})_{n\in\Z} \in \C^\Z, \sum_{n\in \Z} n^2 |a_n|<\infty\Big\}.\]
Now the (equivalence classes of) irreducible representations of $SU_q(2)$ are indexed by half-integers, and each representation $U^s$ (with $s \in \frac{1}{2}\Z_+$) is $(2s+1)$-dimensional. Also, in the notation of \cite[Subsection 2.3]{FST}, for all $t\in\R$ we have $\tau_t(\alpha)=\alpha$ and $\tau_t(\gamma)=|q|^{2it}\gamma$ by \cite[formulas (5.19) and (A 1.3)]{Woronowicz__CMP} (see also \cite[Example 1.7.8]{NeshveyevTuset_book}), so that $\tau_t(U_{ij}^s) =|q|^{(i-j)2it}U_{ij}^s$ and $\tau_{-i/4}(U^s_{ji}) = |q|^{(j-i)/2} U^s_{ji}$ (see \cite[p.~227]{FST}) for $s \in \frac{1}{2}\Z_+$ and $i,j \in \{-s, -s+1 ,\ldots, s-1, s\}$.
Consequently, for $\widetilde{\omega} \in \Lone{\widehat{SU_q(2)}} \cong \ell^1\text{--}\bigoplus_{s \in \frac{1}{2}\Z_+} \ell^1(\matrices_{2s+1})$ we have 
$$\aa{\widetilde\omega} = \tau^\mathrm{u}_{i/4} \Big(\sum_{ s \in \frac{1}{2}\Z_+} \sum_{i,j=-s,\ldots,s} \omega^s_{ij} (U^s_{ji})^*\Big) = \sum_{ s \in \frac{1}{2}\Z_+} \sum_{i,j=-s,\ldots,s} |q|^{\frac{j-i}{2}}\omega_{ij}^s (U_{ji}^s)^*,$$
where the symbol $\ell^1\text{--}\bigoplus$ is meant to indicate the $\ell^1$-direct sum and we view
$\widetilde{\omega}=(\omega^s)_{s \in \frac{1}{2}\Z_+}$ as the direct sum of trace-class matrices.

We need the fact that if $\Theta\colon \CC{SU_q(2)}\to \CC{\mathbb{T}}$ denotes the relevant quotient map, then
\[ \Theta(U_{ij}^s) = \delta_{ij} z^{-2i}\;\;\;\; (\forall_{s \in \frac{1}{2}\Z_+}\forall_{i,j \in \{-s,\ldots,s\}})\]
(see the formulas after \cite[Theorem 5.1]{FST}). Thus, we have
\[
\mu_t^{SU_q(2)} (U_{ij}^s) = \delta_{ij} e^{-4ti^2}\;\;\;\; (\forall_{t \in \R}\forall_{s \in \frac{1}{2}\Z_+}\forall_{i,j \in \{-s,\ldots,s\}}),
\]
and in view of the results discussed in the first part of this example  we obtain the following formula:
\begin{align*} \mathscr{D}_{+}^{SU_q(2)}\cap D(\gamma^{SU_q(2)}) = \Big\{&\sum_{ s \in \frac{1}{2}\Z_+} \sum_{i,j=-s,\ldots,s} q^{\frac{i-j}{2}}\omega_{i,j}^s (U_{j,i}^s)^*: \\ &  \widetilde{\omega} \in \ell^1\text{--}\bigoplus_{s \in \frac{1}{2}\Z_+} \ell^1(\matrices_{2s+1})_+, \sum_{n \in \Z} n^2\Big(\sum_{ s \in \frac{1}{2}\Z_+} \omega_{\frac{n}{2},\frac{n}{2}}^s\Big) < \infty\Big\},\end{align*} 
where the expression $\omega_{\frac{n}{2},\frac{n}{2}}^s$ in the last formula should be understood as equal $0$ whenever $\frac{n}{2}\notin \{-s, \ldots,s\}$.

Next, for a fixed a parameter $\mu \in (0,1)$, let $\G$ be the quantum $E_\mu(2)$ group \cite{Woronowicz_E(2)}. It contains $\QH := \mathbb{T}$ as a (maximal classical) closed quantum subgroup \cite[Theorem 4.3]{KalantarNeufang_groups} (see also \cite[Propositions 2.8.32 and 2.8.36]{Jacobs_phd_thesis}). Write $\overline{\C}^\mu := \{\mu^k z : k \in \Z, z \in \mathbb{T} \} \cup \{0\}$, and define an action $\a$ of $\Z$ on $\Cz{\overline{\C}^\mu}$ by $\a_1(f) := f(\mu\cdot)$ for $f \in \Cz{\overline{\C}^\mu}$. Then we can and will identify $\Cz{E_\mu(2)}$ with $\Cz{\overline{\C}^\mu} \rtimes_\a \Z$ \cite[Proposition 4.1.5]{Jacobs_phd_thesis}. Denoting by $(c^n)_{n\in\Z}$ the canonical unitaries in $\M{\Cz{\overline{\C}^\mu} \rtimes_\a \Z}$, the relevant map $\Theta \colon \Cz{E_\mu(2)} \to \CC{\mathbb{T}}$ is given by 
\[
\Theta(f c^n) = f(0) z^n \;\;\;\;  (\forall_{f \in \Cz{\overline{\C}^\mu}} \forall_{n \in \Z}).
\]
Therefore, the $w^{*}$-continuous convolution semigroup of states on $E_\mu(2)$ associated with the heat semigroup on $\mathbb{T}$ is given by
\[
\mu_t^{E_\mu(2)} (f c^n) = f(0) e^{-tn^2}\;\;\;\; (\forall_{t \in \R} \forall_{f \in \Cz{\overline{\C}^\mu}} \forall_{n\in\Z}),
\]
so its generating functional satisfies $f c^n \in D(\gamma^{E_\mu(2)})$ and $\gamma^{E_\mu(2)}(f c^n) = -n^2 f(0)$ for all $f \in \Cz{\overline{\C}^\mu}$ and $n \in \Z$.
Exhibiting an explicit description of $\mathscr{D}_{+}^{E_\mu(2)}$ and $\mathscr{D}_{+}^{E_\mu(2)} \cap D(\gamma^{E_\mu(2)})$ in terms of the canonical generators of $\Cz{E_\mu(2)}$ is more involved than for $SU_q(2)$ due to the complicated nature of the regular representations of $E_\mu(2)$ \cite[Definition 2.3.9 and Corollary 2.3.14]{Jacobs_phd_thesis}, and is outside the scope of the present paper.
\end{example}

\section{Cocycles} \label{sec:cocycles}

As mentioned in the Introduction, \prettyref{thm:gen_func_domain} opens the way to associate cocycles (i.e.\ $\pi$\textendash $\epsilon$ derivations) to convolution semigroups of states, and to introduce notions such as Gaussianity and L\'evy--Khintchine decompositions and develop their theory (see \cite{Schurmann__white_noise_bialg,Franz_Gerhold_Thom__LK_dec_gen_funct} for these concepts in the algebraic setting). In this section we recall the construction of cocycles, intending to continue the development of the theory in later works. 

 The following algebraic result is well known and follows via a GNS-type construction.
\begin{prop}
	\label{prop:gen_func_to_cocycle}Let $\mathscr{A}$ be a unital $*$-algebra.
	Suppose that $\epsilon$ is a character of $\mathscr{A}$, and that
	$\gamma\colon\mathscr{A}\to\C$ is a linear functional satisfying $\gamma(\one)=0$
	that is hermitian and algebraically conditionally positive in the
	sense that $\gamma(\left\{ a^{*}a: a\in\mathscr{A}\cap\ker\epsilon\right\} )\subseteq[0,\infty)$.
	Then there exists a triple $(\H,\pi,c)$, where $\H$ is an inner
	product space, $\pi$ is a unital representation of $\mathscr{A}$
	on $\H$ and $c\colon\mathscr{A}\to\H$ is a $\pi$\textendash $\epsilon$
	derivation that induces the $\prescript{}{\epsilon}{\C}_{\epsilon}^ {}$-coboundary
	of $\gamma$: it is a linear map satisfying
	\[
	\begin{split}
	c(ab) & =\pi(a)c(b)+c(a)\epsilon(b), \\
	\gamma(b^{*}a) & =\left\langle c(b),c(a)\right\rangle +\gamma(a)\overline{\epsilon(b)}+\epsilon(a)\overline{\gamma(b)}
	\end{split}
	\qquad(\forall_{a,b\in\mathscr{A}}).
	\]
	If, in addition, $\mathscr{A}$ is a unital $*$-subalgebra of some
	unital C$^{*}$-algebra $B$, and if $\gamma$ is conditionally positive
	(in the possibly stricter sense of \prettyref{def:cond_pos}, namely
	$\gamma(\mathscr{A}\cap\ker\epsilon\cap B_{+})\subseteq[0,\infty)$),
	then we can choose $\H$ to be a Hilbert space and $\pi$ to be contractive.
\end{prop}

If we apply \prettyref{prop:gen_func_to_cocycle} to an $\Rant^{\mathrm{u}}$-invariant conditionally positive functional, the resulting cocycle additionally has a symmetry property, which in the context of compact quantum groups was exploited  in \cite{Kyed__cohom_prop_T_QG} and \cite{Das_Franz_Kula_Skalski__one_to_one_corres}. Specifically, in the context of \prettyref{thm:gen_func_domain}, with $B$ being $\CzU{\G}^{\#}$
	and $\mathscr{A}$ being $\linspan((\mathscr{D}_{+}\cap D(\gamma))\cup\{\one\})$,
the  cocycle $c$ resulting from \prettyref{prop:gen_func_to_cocycle} is \emph{real} in the sense that 
	\[
	\left\langle c(\Rant^{\mathrm{u}}(b)^{*}),c(\Rant^{\mathrm{u}}(a)^{*})\right\rangle =\left\langle c(a),c(b)\right\rangle \qquad(\forall_{a,b\in\mathscr{A}}).
	\]
	
Finally we show how the $\pi$\textendash $\epsilon$ derivations considered above in the dual to classical case give rise to the usual cocycles viewed as Hilbert space-valued functions on a group satisfying the suitable cocycle relation.

\begin{example}[compare \prettyref{exa:our_algebra_dual_to_classical_case}]
	Suppose that $G$ is a locally compact group and consider the locally
	compact quantum group $\G:=\widehat{G}$. Denote the left Haar measure
	of $G$ by $\mu$ and the co-unit of $\G$ by $\epsilon$. Set $\mathscr{A}:=\linspan\left(\lambda_{\mathrm{u}}(\Cc G)\cup\left\{ \one\right\} \right)$
	inside $\CStarF G^{\#}$. Fix a unitary representation $\Pi$ of $G$
	on a Hilbert space $\H$ and denote by $\pi$ the representation of
	$\CStarF G$ on $\H$ associated to $\Pi$.
	
	A \emph{(1-) cocycle} of $G$ with respect to $\Pi$ is a continuous
	map $b\colon G\to\H$ such that $b(ts)=\Pi(t)b(s)+b(t)$ for all $t,s\in G$.
	For such a cocycle, the linear map $c\colon \mathscr{A}\to\H$ given by
	$c(\lambda_{\mathrm{u}}(f)):=\int_{G}f(t)b(t)\d \mu(t)$ for $f\in\Cc G$ and $c(\one):=0$
	is well defined by the continuity of $b$, and is a $\pi$\textendash $\epsilon$
	derivation. Furthermore, $c$ satisfies the following continuity property:
	
	$\left(\cent\right)$ For every compact $K\subseteq G$ there exists
	a constant $0\le m_{K}<\infty$ such that $\left\Vert c(\lambda_{\mathrm{u}}(f))\right\Vert \le m_{K}\left\Vert f\right\Vert _{\Lone G}$
	for all $f\in\Cc G$ supported by $K$.
	
	Conversely, suppose that $G$ is second countable (hence $\sigma$-compact)
	and let $c\colon \mathscr{A}\to\H$ be a $\pi$\textendash $\epsilon$ derivation
	satisfying $\left(\cent\right)$. We will prove that it is induced
	by a cocycle of $G$ as above.
	
	For each compact set $K\subseteq G$, let $\Cc{G;K}:=\left\{ f\in\Cc G: f\text{ is supported by }K\right\} $
	and $\Cc{G|K}:=\left\{ f|_{K}:f\in\Cc{G;K}\right\} $. Also consider
	the (finite) restriction of the positive measure space $(G,\text{Borel},\mu)$
	to $K$, and denote by $\Lone K,\Linfty K$ the resulting $L^{1},L^{\infty}$-spaces. 
	
	Let $\mathbb{K}:=\left\{ K\subseteq G: K\text{ is compact and }\Cc{G|K}\text{ is dense in }\Lone K\text{ in the }L^{1}\text{-norm}\right\} $.
	As we show below, in \prettyref{lem:sc_LCG_good_compact_sets}, $\mathbb{K}$
	contains a sequence $\left(K_{n}\right)_{n=1}^{\infty}$ such that
	each compact subset of $G$ is contained in some $K_{n}$.
	
	For every $K\in\mathbb{K}$, $\left(\cent\right)$ implies that the
	map $\Cc{G;K}\ni f\mapsto c(\lambda_{\mathrm{u}}(f))$ induces a bounded
	linear map from $\Lone K$ to $\H$; and $\Lone K$ is separable because
	$K$ is second countable. We deduce that the image of $c$ in $\H$
	is separable, so we may and shall assume that $\H$ is separable.
	
	Take again $K\in\mathbb{K}$. Denote by $\Linfty{K,\H}$ the Banach space of equivalence classes of weakly measurable
	essentially bounded functions from $K$ to $\H$.
	Since $\H$ is separable, the Banach spaces $\Linfty{K,\H}$ and $B(\Lone K,\H)$ are canonically isometrically isomorphic \cite[Theorem VI.8.6]{Dunford_Schwartz_1}.
	Therefore, the bounded map from $\Lone K$ to $\H$ discussed above induces an
	element $b_{K}\in\Linfty{K,\H}$ such that $c(\lambda_{\mathrm{u}}(f))=\int_{K}f(t)b_{K}(t)\d \mu(t)$
	weakly in $\H$ for all $f\in\Cc{G;K}$. Using the sequence $\left(K_{n}\right)_{n=1}^{\infty}$
	in $\mathbb{K}$ we conclude that there exists a weakly measurable
	function $b\colon G\to\H$ that is bounded on each compact subset of $G$
	and satisfies $c(\lambda_{\mathrm{u}}(f))=\int_{G}f(t)b(t)\d \mu(t)$ weakly
	for all $f\in\Cc G$.
	
	The assumption that $c\colon\mathscr{A}\to\H$ is a $\pi$\textendash $\epsilon$
	derivation means that for each $f,g\in\Cc G$,
	\begin{equation}
	c(\lambda_{\mathrm{u}}(f\conv g))=\int_{G}f(t)\Pi(t)c(\lambda_{\mathrm{u}}(g))\d \mu(t)+\left(\int_{G}g(t)\d \mu(t)\right)c(\lambda_{\mathrm{u}}(f))\label{eq:pi_e_derivative__dual_classical}
	\end{equation}
	(the left integral converges in norm). For a function $h\colon G\to\C$
	use the notation $h^{\lor}:=h(\cdot^{-1})$. Let $\z\in\H$, and write
	$b_{\z}:=\left\langle \z,b(\cdot)\right\rangle $. Then
	\[
	\begin{split}\left\langle \z,c(\lambda_{\mathrm{u}}(f\conv g))\right\rangle  & =\int_{G}\left(\int_{G}f(t)g(t^{-1}s)\d \mu(t)\right)b_{\z}(s)\d \mu(s)\\
	& =\int_{G}f(t)\left(\int_{G}g(t^{-1}s)b_{\z}(s)\d \mu(s)\right)\mathrm{d} \mu(t)=\int_{G}f(t)(g\conv b_{\z}^{\lor})(t^{-1})\d \mu(t).
	\end{split}
	\]
	As a result, \prettyref{eq:pi_e_derivative__dual_classical} implies
	that the following equality holds almost everywhere: 
	\begin{equation}
	(g\conv b_{\z}^{\lor})^{\lor}=\left\langle \z,\Pi(\cdot)c(\lambda_{\mathrm{u}}(g))\right\rangle +\left(\int_{G}g(t)\d \mu(t)\right)b_{\z}.\label{eq:pi_e_derivative__dual_classical__2}
	\end{equation}
	Notice that since $b_{\z}$ is bounded on compact sets, $(g\conv b_{\z}^{\lor})^{\lor}$
	is continuous (for the convolution of an $L^{1}$ function and an
	$L^{\infty}$ function is continuous); and evidently so is $\left\langle \z,\Pi(\cdot)c(\lambda_{\mathrm{u}}(g))\right\rangle $.
	Choosing $g$ such that $\int_{G}g(t)\d \mu(t)\neq0$ we deduce that $b_{\z}$
	is equal almost everywhere to a continuous function. Since $b$ is
	bounded on compact sets, since the complement of a $\mu$-null set
	is dense, and since $\H$ is separable, this yields that $b$ is equal
	almost everywhere to a weakly continuous function,
	so that we may and shall assume that $b$ itself is weakly continuous.
	Therefore, \prettyref{eq:pi_e_derivative__dual_classical__2} holds
	everywhere for all $g\in\Cc G$ and $\z\in\H$. 
	 This implies, by a simple calculation using weak continuity of $b$, that $b(ts)=\Pi(t)b(s)+b(t)$
	for all $t,s\in G$. Finally, as $\H$ is separable, $b$ is continuous
	(in norm) by \citep[Exercise 2.14.3]{Bekka_de_la_Harpe_Valette__book}.
	In conclusion, $b$ is a cocycle.
\end{example}

\begin{lem}
	\label{lem:sc_LCG_good_compact_sets}For a second countable, locally
	compact group $G$, the set $\mathbb{K}$ defined above contains a
	sequence such that each compact subset of $G$ is contained in some
	element of this sequence.
\end{lem}

\begin{proof}
	For a compact $K\subseteq G$, $\CC K$ is dense in $\Lone K$, hence
	$K\in\mathbb{K}$ (namely, $\Cc{G|K}$ is dense in $\Lone K$) if
	and only if $\Cc{G|K}$ is dense in $\CC K$ (in the $L^{1}$-norm).
	This is equivalent to $1_{K}$ lying in the closure of $\Cc{G|K}$,
	because $\Cc{G|K}$ is an ideal in $\CC K$ by Tietze's theorem.
	
	Since $G$ is second countable, there exists a (left-invariant) metric
	$d$ on $G$ that induces the topology on $G$ such that each open
	$d$-ball has compact closure \citep{Struble__metrics_LCGs}. Denote again
	the left Haar measure of $G$ by $\mu$, and for $r>0$ write $B_{r}$
	for the open $d$-ball around $e$ of radius $r$. Since the function
	$\left(0,\infty\right)\to\left(0,\infty\right)$ given by $r\mapsto\mu(B_{r})$
	is (well defined and) non-decreasing, it admits arbitrarily large
	points of continuity. Thus, it suffices to prove that if $r>0$ is
	such a point then $\overline{B_{r}}\in\mathbb{K}$. For every $0<\delta<r$
	there exists by Urysohn's lemma $f_{\delta}\in\Cc G$ with values
	in $\left[0,1\right]$ satisfying $f_{\delta}|_{\overline{B_{r-\delta}}}\equiv1$
	and $\supp f_{\delta}\subseteq B_{r}$. Then $f_{\delta}\in\Cc{G;\overline{B_{r}}}$,
	hence $f_{\delta}|_{\overline{B_{r}}}\in\Cc{G|\overline{B_{r}}}$.
	We have $\left\{ x\in\overline{B_{r}}: f_{\delta}(x)\neq1\right\} \subseteq\overline{B_{r}}\backslash\overline{B_{r-\delta}}\subseteq B_{r+\delta}\backslash B_{r-\delta}$,
	so that 
	\[
	\Vert f_{\delta}|_{\overline{B_{r}}}-1_{\overline{B_{r}}}\Vert_{\Lone{\overline{B_{r}}}}\le\mu(B_{r+\delta})-\mu(B_{r-\delta})\xrightarrow[\delta\to0^{+}]{}0
	\]
	by assumption. This completes the proof. 
\end{proof}

\section{\label{sec:reconstruction}Reconstructing convolution semigroups
from generating functionals}

An important consequence of the celebrated Sch\"urmann reconstruction
theorem \citep{Schurmann__pos_and_cond_pos_coalg,Schurmann__white_noise_bialg}
is that for every conditionally positive, hermitian functional $\gamma$
on a $*$-bialgebra that annihilates the unit there exists a (unique)
$w^{*}$-continuous convolution semigroup of states $\left(\mu_{t}\right)_{t\ge0}$
such that $\mu_{t}=\exp_{\star}(t\gamma)$ for all $t\ge0$. In
particular, this theorem applies to compact quantum groups. In this
section we establish a reconstruction theorem for arbitrary locally
compact quantum groups under a symmetry assumption.
\begin{notation}
Let $n\in\N$. In the next results we use the convention that for
a Hilbert space $\H$ we write vectors $\z\in\Ltwo{\matrices_{n},\Tr_{n}}\tensor\H$
as matrices $\left(\z_{ij}\right)_{1\le i,j\le n}\in \matrices_{n}(\H)$ with
respect to some fixed orthonormal basis. Furthermore, we use the notation $\pi^{(n)}$ defined in the Introduction for the pair $(\Linfty{\G},\varphi)$. So $\pi^{(n)}$ is the nearest-point projection of $\Ltwo{\matrices_{n},\Tr_{n}}\tensor\Ltwo{\G}$
onto the key closed convex set associated with $(\matrices_{n}\tensor\Linfty{\G},\Tr_{n}\tensor\varphi)$.
\end{notation}

\begin{lem}
\label{lem:id_plus_R__e_minus_i__ineq}Let $n\in\N$. Then for every
$\z\in\Ltwo{\matrices_{n},\Tr_{n}}\tensor\Ltwo{\G}$ we have 
\[
\sum_{i,j=1}^{n}\left[(\i+\Rant^{\mathrm{u}})(\epsilon(\cdot)\one-\i)\right](\aa{\widehat{\om}_{\pi^{(n)}(\z)_{i,j}}})\le\sum_{i,j=1}^{n}\left[(\i+\Rant^{\mathrm{u}})(\epsilon(\cdot)\one-\i)\right](\aa{\widehat{\om}_{\z_{i,j}}})
\]
in $\CzU{\G}^{\#}$. Recall that the operators on both sides of this
inequality are positive by \prettyref{lem:D_plus}~\prettyref{enu:D_plus__3}.
\end{lem}

\begin{proof}
Fix $\z\in\Ltwo{\matrices_{n},\Tr_{n}}\tensor\Ltwo{\G}$. Let $\nu$ be an
$\Rant^{\mathrm{u}}$-invariant state of $\CzU{\G}$. Since the map
$R_{\nu}$ on $\Linfty{\G}$ is KMS-symmetric with respect to $\varphi$
\citep[Corollary 2.8]{Skalski_Viselter__convolution_semigroups} and
completely Markov, the map $\widetilde{R}_{\nu}^{(2,\varphi)}$ on
$\Ltwo{\G}$ is a (contractive) selfadjoint completely Markov operator
with respect to $\varphi$; see \citep[Appendix]{Skalski_Viselter__convolution_semigroups}
for the terminology.
By \citep[Lemma 5.2]{Goldstein_Lindsay__Markov_sgs_KMS_symm_weight}
the quadratic form on $\Ltwo{\matrices_{n},\Tr_{n}}\tensor\Ltwo{\G}$ associated
with $\one_{\matrices_{n}}\tensor(\one-\widetilde{R}_{\nu}^{(2,\varphi)})$
is Dirichlet with respect to $\Tr_{n}\tensor\varphi$, so that 
\[
\widehat{\omega}_{\pi^{(n)}(\z)}\bigl(\one_{\matrices_{n}}\tensor(\one-\widetilde{R}_{\nu}^{(2,\varphi)})\bigr)\le\widehat{\omega}_{\z}\bigl(\one_{\matrices_{n}}\tensor(\one-\widetilde{R}_{\nu}^{(2,\varphi)})\bigr),
\]
that is, 
\[
\sum_{i,j=1}^{n}\widehat{\omega}_{\pi^{(n)}(\z)_{i,j}}(\one-\widetilde{R}_{\nu}^{(2,\varphi)})\le\sum_{i,j=1}^{n}\widehat{\omega}_{\z_{i,j}}(\one-\widetilde{R}_{\nu}^{(2,\varphi)}).
\]
From \prettyref{prop:SV_prop_3_8}~\prettyref{enu:SV_prop_3_8__1} and \prettyref{lem:D_plus}~\prettyref{enu:D_plus__2}
this is equivalent to 
\[
\sum_{i,j=1}^{n}(\epsilon-\nu)(\aa{\widehat{\om}_{\pi^{(n)}(\z)_{i,j}}})\le\sum_{i,j=1}^{n}(\epsilon-\nu)(\aa{\widehat{\om}_{\z_{i,j}}}).
\]
Let now $\mu$ be an arbitrary state of $\CzU{\G}$. Then, as $\nu:=\frac{1}{2}(\mu+\mu\circ\Rant^{\mathrm{u}})$
is an $\Rant^{\mathrm{u}}$-invariant state of $\CzU{\G}$, we
deduce from the last formula that
\[
\mu\bigl(\sum_{i,j=1}^{n}\left[(\i+\Rant^{\mathrm{u}})(\epsilon(\cdot)\one-\i)\right](\aa{\widehat{\om}_{\pi^{(n)}(\z)_{i,j}}})\bigr)\le\mu\bigl(\sum_{i,j=1}^{n}\left[(\i+\Rant^{\mathrm{u}})(\epsilon(\cdot)\one-\i)\right](\aa{\widehat{\om}_{\z_{i,j}}})\bigr),
\]
and the assertion follows.
\end{proof}

In this section we will consider conditionally positive functionals as in \prettyref{def:cond_pos}
with $B$ being $\CzU{\G}^{\#}$ and $\epsilon$ being the co-unit.

\begin{cor}[compare \prettyref{cor:gamma_domain_domination}]
\label{cor:general_gen_func__D_plus}Let $\mathscr{A}$ be a globally
$\Rant^{\mathrm{u}}$-invariant unital subspace of $\CzU{\G}^{\#}$
and $\gamma\colon \mathscr{A}\to\C$ a linear functional satisfying $\gamma(\one)=0$
that is $\Rant^{\mathrm{u}}$-invariant and conditionally positive. 
\begin{enumerate}
\item \label{enu:general_gen_func__D_plus__0}For every $a\in\mathscr{A}$, if  $\left[(\i+\Rant^{\mathrm{u}})(\epsilon(\cdot)\one-\i)\right](a) \ge 0$ in $\CzU{\G}^{\#}$, then $-\gamma(a) \ge 0$.
\item \label{enu:general_gen_func__D_plus__1}For every $\widehat{\om}\in\Lone{\widehat{\G}}_{+}$
such that $\aa{\widehat{\om}}\in\mathscr{A}$ we have $-\gamma(\aa{\widehat{\om}})\ge0$. 
\item \label{enu:general_gen_func__D_plus__2}For every $n\in\N$ and $\z\in\Ltwo{\matrices_{n},\Tr_{n}}\tensor\Ltwo{\G}$
such that $\sum_{i,j=1}^{n}\aa{\widehat{\om}_{\z_{i,j}}},\sum_{i,j=1}^{n}\aa{\widehat{\om}_{\pi^{(n)}(\z)_{i,j}}}\in \mathscr{A}$
we have $-\gamma(\sum_{i,j=1}^{n}\aa{\widehat{\om}_{\pi^{(n)}(\z)_{i,j}}})\le-\gamma(\sum_{i,j=1}^{n}\aa{\widehat{\om}_{\z_{i,j}}})$.
\end{enumerate}
\end{cor}

\begin{proof}
\prettyref{enu:general_gen_func__D_plus__0}  We have $\left[(\i+\Rant^{\mathrm{u}})(\epsilon(\cdot)\one-\i)\right](a)\in\mathscr{A}\cap\ker\epsilon$
and $\gamma\left\{ \left[(\i+\Rant^{\mathrm{u}})(\epsilon(\cdot)\one-\i)\right](a)\right\} =-2\gamma(a)$ for every $a\in\mathscr{A}$.
The assertion thus follows from conditional positivity of $\gamma$.

\prettyref{enu:general_gen_func__D_plus__1} Combine \prettyref{enu:general_gen_func__D_plus__0} with \prettyref{lem:D_plus}~\prettyref{enu:D_plus__3}.

\prettyref{enu:general_gen_func__D_plus__2} Combine \prettyref{enu:general_gen_func__D_plus__0} with \prettyref{lem:id_plus_R__e_minus_i__ineq}.
\end{proof}

The next theorem is (a first incarnation of) the main result of this section. The motivation for condition \prettyref{enu:conv_smgrp_from_CND_func__4}
in the next theorem is \prettyref{cor:general_gen_func__D_plus}~\prettyref{enu:general_gen_func__D_plus__2}.
See more below.
\begin{thm}
\label{thm:conv_smgrp_from_CND_func}Let $\G$ be a locally compact
quantum group and $\mathscr{A}$ be a globally $\Rant^{\mathrm{u}}$-invariant
unital subspace of $\CzU{\G}^{\#}$. Let $\gamma\colon\mathscr{A}\to\C$
be a linear functional satisfying $\gamma(\one)=0$ that is $\Rant^{\mathrm{u}}$-invariant
and conditionally positive. Assume further that:
\begin{enumerate}[label=\textnormal{(\Roman*)}]
\item \label{enu:conv_smgrp_from_CND_func__1}$\mathscr{A}=\linspan((\mathscr{D}_{+}\cap\mathscr{A})\cup\left\{ \one\right\} )$;
\item \label{enu:conv_smgrp_from_CND_func__2}$\{\z\in\Ltwo{\G}: \aa{\widehat{\om}_{\z}}\in\mathscr{A}\} $
is a dense subspace of $\Ltwo{\G}$;
\item \label{enu:conv_smgrp_from_CND_func__3}$\gamma$ satisfies the following
lower semi-continuity property: if $\left(a_{k}\right)_{k=1}^{\infty}$
is a sequence in $\mathscr{D}_{+}\cap\mathscr{A}$ converging in the
norm of $\CzU{\G}$ to some $a\in\mathscr{D}_{+}\cap\mathscr{A}$,
then $-\gamma(a)\le\liminf_{k\to\infty}(-\gamma(a_{k}))$ (recall
that all these numbers are non-negative by \prettyref{cor:general_gen_func__D_plus}~\prettyref{enu:general_gen_func__D_plus__1});
\item \label{enu:conv_smgrp_from_CND_func__4}for every $n\in\N$ and $\z\in\Ltwo{\matrices_{n},\Tr_{n}}\tensor\Ltwo{\G}$
such that $\aa{\widehat{\om}_{\z_{i,j}}}\in\mathscr{A}$ for each
$1\le i,j\le n$ there exists a sequence $\left(\eta^{k}\right)_{k=1}^\infty$
in $\Ltwo{\matrices_{n},\Tr_{n}}\tensor\Ltwo{\G}$ that converges to $\pi^{(n)}(\z)$
such that $\aa{\widehat{\om}_{\eta_{i,j}^{k}}}\in\mathscr{A}$ for
each $k\in\N$ and $1\le i,j\le n$ and $\liminf_{k\to\infty}\sum_{i,j=1}^{n}(-\gamma(\aa{\widehat{\om}_{\eta_{i,j}^{k}}}))\le\sum_{i,j=1}^{n}(-\gamma(\aa{\widehat{\om}_{\z_{i,j}}}))$.
\end{enumerate}
Then there exists a $w^{*}$-continuous convolution semigroup of $\Rant^{\mathrm{u}}$-invariant
contractive positive functionals on $\CzU{\G}$ whose generating functional extends $\gamma$ on $\CzU \G$.
\end{thm}

\begin{proof}
The set $D:=\{ \z\in\Ltwo{\G}:\aa{\widehat{\om}_{\z}}\in\mathscr{A}\} $
is a dense subspace of $\Ltwo{\G}$ by \ref{enu:conv_smgrp_from_CND_func__2}.
\prettyref{cor:general_gen_func__D_plus}~\prettyref{enu:general_gen_func__D_plus__1}
allows defining a map $Q\colon D\to[0,\infty)$ by $Q(\z):=-\gamma(\aa{\widehat{\om}_{\z}})$
for $\z\in D$. Then $Q$ is a densely-defined quadratic form. By
\citep[Proposition A.9]{Stratila__mod_thy}, $Q$ is closable, because
if $\left(\z_{k}\right)_{k=1}^{\infty}$ is a sequence in $D$ that
converges to $\z\in D$, then $\bigl\Vert\aa{\widehat{\om}_{\z_{k}}}-\aa{\widehat{\om}_{\z}}\bigr\Vert=\bigl\Vert\aa{\widehat{\om}_{\z_{k}}-\widehat{\om}_{\z}}\bigr\Vert\le\left\Vert \widehat{\om}_{\z_{k}}-\widehat{\om}_{\z}\right\Vert \xrightarrow[k\to\infty]{}0$,
so $Q(\z)=-\gamma(\aa{\widehat{\om}_{\z}})\le\liminf_{k\to\infty}(-\gamma(\aa{\widehat{\om}_{\z_{k}}}))=\liminf_{k\to\infty}Q(\z_{k})$
by \ref{enu:conv_smgrp_from_CND_func__3}. Recall that $D(\overline{Q})$
consists of all $\z\in\Ltwo{\G}$ for which there is a sequence $\left(\z_{k}\right)_{k=1}^{\infty}$
in $D$ with $\z_{k}\xrightarrow[k\to\infty]{}\z$ and $Q(\z_{k}-\z_{\ell})\xrightarrow[k,\ell\to\infty]{}0$,
in which case $\lim_{k\to\infty}Q(\z_{k})$ exists and equals $\overline{Q}(\z)$.
Since $Q$ is obviously invariant under $\mathcal{U}(\Linfty{\widehat{\G}}')$,
so is $\overline{Q}$. 

We now show that $\overline{Q}$ is completely Dirichlet with respect
to $\varphi$. Fix $n\in\N$ and let $\left(\z_{i,j}\right)_{i,j=1}^{n}=\z\in D(Q^{(n)})$.
By \prettyref{enu:conv_smgrp_from_CND_func__4} there is a sequence $\left(\eta^{k}\right)_{k=1}^\infty$
in $\Ltwo{\matrices_{n},\Tr_{n}}\tensor\Ltwo{\G}$ that converges to $\pi^{(n)}(\z)$
such that $\eta_{i,j}^{k}\in D$ for every $k\in\N$ and
$1\le i,j\le n$ and $\liminf_{k\to\infty}\sum_{i,j=1}^{n}(-\gamma(\aa{\widehat{\om}_{\eta_{i,j}^{k}}}))\le\sum_{i,j=1}^{n}(-\gamma(\aa{\widehat{\om}_{\z_{i,j}}}))$.
Hence, the lower semi-continuity of $\overline{Q}^{(n)}$ implies
that 
\[
\overline{Q}^{(n)}(\pi^{(n)}(\z))\le\liminf_{k\to\infty}Q^{(n)}(\eta^{k})=\liminf_{k\to\infty}\sum_{i,j=1}^{n}(-\gamma(\aa{\widehat{\om}_{\eta_{i,j}^{k}}}))\le\sum_{i,j=1}^{n}(-\gamma(\aa{\widehat{\om}_{\z_{i,j}}}))=Q^{(n)}(\z).
\]
For the general case, take $\left(\z_{i,j}\right)_{i,j=1}^{n}=\z\in D(\overline{Q}^{(n)})$,
and pick a sequence $\bigl(\big(\z_{i,j}^{k}\big)_{i,j=1}^{n}\bigr)_{k=1}^{\infty}=\left(\z^{k}\right)_{k=1}^{\infty}$
in $D(Q^{(n)})$ such that $\z^{k}\xrightarrow[k\to\infty]{}\z$ and
$Q(\z_{i,j}^{k}-\z_{i,j}^{\ell})\xrightarrow[k,\ell\to\infty]{}0$
for each $1\le i,j\le n$. The foregoing, the continuity of $\pi^{(n)}$
and the lower semi-continuity of $\overline{Q}^{(n)}$ imply that
\[
\overline{Q}^{(n)}(\pi^{(n)}(\z))\le\liminf_{k\to\infty}\overline{Q}^{(n)}(\pi^{(n)}(\z^{k}))\le\lim_{k\to\infty}Q^{(n)}(\z^{k})=\overline{Q}^{(n)}(\z).
\]
This proves that $\overline{Q}^{(n)}$ is Dirichlet with respect
to $\Tr_{n}\tensor\varphi$. Since $n\in\N$ was arbitrary, $\overline{Q}$
is completely Dirichlet with respect to $\varphi$.

We conclude from \prettyref{thm:SV_thm_3_4}
that there exists a $w^{*}$-continuous convolution semigroup $\left(\mu_{t}\right)_{t\ge0}$
of $\Rant^{\mathrm{u}}$-invariant contractive positive functionals on $\CzU{\G}$ whose associated
completely Dirichlet form is $\overline{Q}$. Denoting the generating
functional of $\left(\mu_{t}\right)_{t\ge0}$ by $\gamma'$, we obtain
$\gamma|_{\mathscr{D}_{+}\cap\mathscr{A}}\subseteq\gamma'$ from \prettyref{prop:SV_prop_3_8}~\prettyref{enu:SV_prop_3_8__2},
hence $\gamma|_{\CzU \G \cap\mathscr{A}}\subseteq\gamma'$ by \ref{enu:conv_smgrp_from_CND_func__1}.
\end{proof}

The disadvantage of \prettyref{thm:conv_smgrp_from_CND_func} is that, in principle, the functional $\gamma$ may extend to two different generating functionals (so it does not determine the convolution semigroup of positive functionals in question uniquely). This cannot happen if we strengthen condition~\prettyref{enu:conv_smgrp_from_CND_func__3}, as we show in the next theorem.

\begin{thm} \label{thm:conv_smgrp_from_CND_func_unique}
	Let $\G$ be a locally compact
	quantum group and let $\mathscr{A}$ be a globally $\Rant^{\mathrm{u}}$-invariant
	unital subspace of $\CzU{\G}^{\#}$. Let $\gamma\colon \mathscr{A}\to\C$
	be a linear functional satisfying $\gamma(\one)=0$ that is $\Rant^{\mathrm{u}}$-invariant
	and conditionally positive. Assume further that conditions  \prettyref{enu:conv_smgrp_from_CND_func__1},  
	\prettyref{enu:conv_smgrp_from_CND_func__2} and \prettyref{enu:conv_smgrp_from_CND_func__4} from Theorem \ref{thm:conv_smgrp_from_CND_func} hold, and that we have a stronger version of condition~\prettyref{enu:conv_smgrp_from_CND_func__3}, namely
		\begin{enumerate}[start=3,label=\textnormal{(\Roman*.a)}]
		\item \label{enu:conv_smgrp_from_CND_func__3_strong}if $\left(a_{k}\right)_{k=1}^{\infty}$
		is a sequence in $\mathscr{D}_{+}\cap\mathscr{A}$ converging in the
		norm of $\CzU{\G}$ to some $a\in\mathscr{D}_{+}$ and $\liminf_{k\to\infty}(-\gamma(a_{k}))<\infty$,
		then $a\in\mathscr{A}$ and $-\gamma(a)\le\liminf_{k\to\infty}(-\gamma(a_{k}))$.
	\end{enumerate}
Then there exists a \emph{unique} $w^{*}$-continuous convolution semigroup of $\Rant^{\mathrm{u}}$-invariant
contractive positive functionals on $\CzU{\G}$ whose generating functional $\gamma'$ satisfies $\mathscr{D}_{+}\cap D(\gamma')=\mathscr{D}_{+}\cap\mathscr{A}$ and extends $\gamma$ on $\CzU \G$.
\end{thm}
\begin{proof}
	We proceed as in the proof of Theorem \ref{thm:conv_smgrp_from_CND_func}. The form $Q$ defined there is closed, for if $\left(\z_{k}\right)_{k=1}^{\infty}$
	is a sequence in $D$ that converges to $\z\in\Ltwo{\G}$ such that
	$\liminf_{k\to\infty}Q(\z_{k})<\infty$, namely $\liminf_{k\to\infty}(-\gamma(\aa{\widehat{\om}_{\z_{k}}}))<\infty$,
	then as above we have $\aa{\widehat{\om}_{\z_{k}}}\xrightarrow[k\to\infty]{}\aa{\widehat{\om}_{\z}}$
	in norm, so by assumption $\aa{\widehat{\om}_{\z}}\in\mathscr{A}$
	(that is, $\z\in D$) and $Q(\z)=-\gamma(\aa{\widehat{\om}_{\z}})\le\liminf_{k\to\infty}(-\gamma(\aa{\widehat{\om}_{\z_{k}}}))=\liminf_{k\to\infty}Q(\z_{k})$.
	Hence, $\mathscr{D}_{+}\cap D(\gamma')=\mathscr{D}_{+}\cap\mathscr{A}$
	by \prettyref{prop:SV_prop_3_8}~\prettyref{enu:SV_prop_3_8__2}.
	The generating functional $\gamma''$ of any other $w^{*}$-continuous
	convolution semigroup of $\Rant^{\mathrm{u}}$-invariant contractive positive functionals on
	$\CzU{\G}$ such that $\gamma''$ extends $\gamma$ on $\CzU \G$ and $\mathscr{D}_{+}\cap D(\gamma'')=\mathscr{D}_{+}\cap\mathscr{A}$
	behaves just like $\gamma'$ on $\mathscr{D}_{+}$, and so $\gamma'=\gamma''$
	by  \prettyref{prop:gen_func_conv_semi_detrm_by_D_plus}. 
\end{proof}

\begin{rem}\label{rem:reconst_model_situation}
	The conditions of \prettyref{thm:conv_smgrp_from_CND_func_unique} are fulfilled
	in the `model' situation, when $\gamma$ is the generating functional
	of a $w^{*}$-continuous convolution semigroup $\left(\mu_{t}\right)_{t\ge0}$
	of $\Rant^{\mathrm{u}}$-invariant states on $\G$ and we take $\mathscr{A}_{\gamma}:=\linspan((\mathscr{D}_{+}\cap D(\gamma))\cup\{\one\})\subseteq\CzU{\G}^{\#}$
	for $\mathscr{A}$, to which $\gamma$ is extended by making it vanish at $\one$, and the (unique, as indicated) convolution semigroup
	constructed in the theorem is again $\left(\mu_{t}\right)_{t\ge0}$.
	Indeed, $\mathscr{A}_{\gamma}$ is globally $\Rant^{\mathrm{u}}$-invariant
	by \prettyref{lem:D_plus}. Condition~\prettyref{enu:conv_smgrp_from_CND_func__1}
	plainly holds; notice that $\mathscr{D}_{+}\cap\mathscr{A}_{\gamma}=\mathscr{D}_{+}\cap D(\gamma)$.
	Condition~\prettyref{enu:conv_smgrp_from_CND_func__2} holds by \prettyref{prop:SV_prop_3_8}~\prettyref{enu:SV_prop_3_8__2}.
	Condition~\prettyref{enu:conv_smgrp_from_CND_func__3_strong} follows
	from \prettyref{prop:gamma_domain}~\prettyref{enu:gamma_domain__3}.
	And condition~\prettyref{enu:conv_smgrp_from_CND_func__4} verifies
	as for $n\in\N$ and $\z\in\Ltwo{\matrices_{n},\Tr_{n}}\tensor\Ltwo{\G}$ such
	that $\aa{\widehat{\om}_{\z_{i,j}}}\in D(\gamma)$ for each $1\le i,j\le n$
	one can just consider the vector $\pi^{(n)}(\z)\in\Ltwo{\matrices_{n},\Tr_{n}}\tensor\Ltwo{\G}$
	itself: indeed, combining \prettyref{lem:id_plus_R__e_minus_i__ineq}
	and \prettyref{cor:gamma_domain_domination} we get that $\aa{\sum_{i,j=1}^{n}\widehat{\om}_{\pi^{(n)}(\z)_{i,j}}}\in D(\gamma)$,
	equivalently $\aa{\widehat{\om}_{\pi^{(n)}(\z)_{i,j}}}\in D(\gamma)$
	for all $1\le i,j\le n$, and the desired inequality follows from
	\prettyref{cor:general_gen_func__D_plus}~\prettyref{enu:general_gen_func__D_plus__2}.
\end{rem}

\begin{rem}\label{rem:conv_smgrp_from_CND_func__value_at_1}
Assume that $\G$ is not compact. The fact that the convolution semigroup obtained in \prettyref{thm:conv_smgrp_from_CND_func} does not necessarily consist of states may seem counter-intuitive, but it is easy to explain. Let $\left(\mu_{t}\right)_{t\ge0}$ be a $w^{*}$-continuous convolution semigroup of $\Rant^{\mathrm{u}}$-invariant states of $\CzU{\G}$, and denote its generating functional by $\gamma_\mathrm{s}'$. Fix $c\ge0$, and let $\gamma'$ be the generating functional of the convolution semigroup ${(e^{-ct}\mu_{t})}_{t\ge0}$, namely $\gamma'=\gamma_\mathrm{s}'-c\epsilon$. Now, extend $\gamma'$ to a linear functional $\gamma$ on $\linspan (D(\gamma_\mathrm{s}') \cup \{\one\}) \subseteq \CzU{\G}^{\#}$ by making it vanish at $\one$. We assert that $\gamma$ is conditionally positive. Indeed, let $\lambda \one + a \in D(\gamma) \cap \ker \epsilon \cap \CzU{\G}^{\#}_+$ ($\lambda \in \C$, $a \in D(\gamma_\mathrm{s}') \subseteq \CzU{\G}$). Since $\CzU{\G}$ is not unital, we have $0 \in \sigma(a)$, thus $\lambda \ge 0$. Furthermore, $\lambda = -\epsilon(a)$. All in all, using the fact that the (natural) extension of $\gamma_\mathrm{s}'$ to $\linspan (D(\gamma_\mathrm{s}') \cup \{\one\})$ vanishing at $\one$ (also denoted by $\gamma_\mathrm{s}'$ in the next equation) is conditionally positive, we have
\[
\gamma(\lambda \one + a) = \gamma'(a) = \gamma_\mathrm{s}'(a)-c\epsilon(a) = \gamma_\mathrm{s}'(\lambda\one + a) +c\lambda \ge 0
\]
as the sum of two non-negative numbers.

The restriction of $\gamma$ to $\mathscr{A}_{\gamma}:=\linspan((\mathscr{D}_{+}\cap D(\gamma))\cup\{\one\})$ now satisfies the conditions of \prettyref{thm:conv_smgrp_from_CND_func_unique}, and the constructed convolution semigroup is ${(e^{-ct}\mu_{t})}_{t\ge0}$: this follows by arguing as in the previous remark.
\end{rem}

\begin{rem}\label{rem:symm_cond_pos_auto_herm}
An $\Rant^{\mathrm{u}}$-invariant conditionally positive linear functional $\gamma\colon \mathscr{A}\to\C$
with $\gamma(\one)=0$, where $\mathscr{A}$ is a globally $\Rant^{\mathrm{u}}$-invariant
unital subspace of $\CzU{\G}^{\#}$ satisfying condition~\prettyref{enu:conv_smgrp_from_CND_func__1}
of \prettyref{thm:conv_smgrp_from_CND_func}, is automatically hermitian,
because it is hermitian on $\mathscr{D}_{+}\cap\mathscr{A}$: if $\widehat{\omega}\in\Lone{\widehat{\G}}_{+}$
and $\aa{\widehat{\omega}}\in\mathscr{A}$, then $\aa{\widehat{\omega}}^{*}=\Rant^{\mathrm{u}}(\aa{\widehat{\om}})$
by \prettyref{lem:D_plus}~\prettyref{enu:D_plus__1}, thus $\gamma(\aa{\widehat{\omega}}^{*})=\gamma(\aa{\widehat{\om}})$,
and this number is non-positive, and in particular real, by \prettyref{cor:general_gen_func__D_plus}~\prettyref{enu:general_gen_func__D_plus__1}.
\end{rem}

\begin{rem}
	The difference between \prettyref{thm:conv_smgrp_from_CND_func} and \prettyref{thm:conv_smgrp_from_CND_func_unique} raises the following question. Is it indeed possible that one can have a strict containment of two noncommutative translation-invariant (completely) Dirichlet forms (in the sense of \prettyref{thm:SV_thm_3_4}~\prettyref{enu:SV_thm_3_4__3})? Classically the answer is negative, as any Dirichlet form as above contains in its domain the algebra $\CcX{2,l}G$, and the L\'evy--Khintchine formula shows that the restriction of the form to this algebra determines the convolution semigroup (so also the Dirichlet form in question). It is worth noting that once we drop the translation invariance, even classically one can construct Dirichlet forms strictly contained in each other, as can be seen for example in \cite{Fitzsimmons}. Note that a similar question can be asked about proper containment of generating functionals.
\end{rem}

Since the appearance of $\pi^{(n)}$ in the above condition~\prettyref{enu:conv_smgrp_from_CND_func__4}
is not desirable, let us observe that it can easily be replaced by
stronger conditions, one of which (condition~\prettyref{enu:conv_smgrp_from_CND_func__4_b}) depends only on $\mathscr{A}$ and not on the values of $\gamma$.

\begin{cor}
\label{cor:conv_smgrp_from_CND_func}\prettyref{thm:conv_smgrp_from_CND_func}
remains true when condition~\prettyref{enu:conv_smgrp_from_CND_func__4} is replaced
by either of the following ones.
\begin{enumerate}[label=\textnormal{(IV.\alph*)}]
\item \label{enu:conv_smgrp_from_CND_func__4_a}For every $n\in\N$ and
$\z,\eta\in\Ltwo{\matrices_{n},\Tr_{n}}\tensor\Ltwo{\G}$ such that $\aa{\widehat{\om}_{\z_{i,j}}}\in\mathscr{A}$
for each $1\le i,j\le n$ and $\sum_{i,j=1}^{n}\left[(\i+\Rant^{\mathrm{u}})(\epsilon(\cdot)\one-\i)\right](\aa{\widehat{\om}_{\eta_{i,j}}})\le\sum_{i,j=1}^{n}\left[(\i+\Rant^{\mathrm{u}})(\epsilon(\cdot)\one-\i)\right](\aa{\widehat{\om}_{\z_{i,j}}})$ in $\CzU{\G}^{\#}$
there exists a sequence $\left(\eta^{k}\right)_{k=1}^\infty$ in $\Ltwo{\matrices_{n},\Tr_{n}}\tensor\Ltwo{\G}$
that converges to $\eta$ such that $\aa{\widehat{\om}_{\eta_{i,j}^{k}}}\in\mathscr{A}$
for each $k\in\N$ and $1\le i,j\le n$ and $\liminf_{k\to\infty}\sum_{i,j=1}^{n}(-\gamma(\aa{\widehat{\om}_{\eta_{i,j}^{k}}}))\le\sum_{i,j=1}^{n}(-\gamma(\aa{\widehat{\om}_{\z_{i,j}}}))$.
\item \label{enu:conv_smgrp_from_CND_func__4_b}For every $\eta\in\Ltwo{\G}$
there exists a sequence $\left(\eta^{k}\right)_{k=1}^\infty$ in $\Ltwo{\G}$
that converges to $\eta$ such that for each $k\in\N$ we
have $\aa{\widehat{\om}_{\eta^{k}}}\in\mathscr{A}$ and \[\left[(\i+\Rant^{\mathrm{u}})(\epsilon(\cdot)\one-\i)\right](\aa{\widehat{\om}_{\eta^{k}}})\le\left[(\i+\Rant^{\mathrm{u}})(\epsilon(\cdot)\one-\i)\right](\aa{\widehat{\om}_{\eta}}) \text{ in }\CzU{\G}^{\#}.\]
\end{enumerate}
\end{cor}

\begin{proof}
We have \prettyref{enu:conv_smgrp_from_CND_func__4_b}$\implies$\prettyref{enu:conv_smgrp_from_CND_func__4_a} by \prettyref{cor:general_gen_func__D_plus}~\prettyref{enu:general_gen_func__D_plus__0},
and \prettyref{enu:conv_smgrp_from_CND_func__4_a}$\implies$\prettyref{enu:conv_smgrp_from_CND_func__4}
by \prettyref{lem:id_plus_R__e_minus_i__ineq}.
\end{proof}
We close with a result connecting Property~(T) and conditionally
positive functionals. One of the main results of \citep{Skalski_Viselter__convolution_semigroups}
says that if $\G$ is second countable and $\widehat{\G}$ does not
have Property~(T), then there exists a $w^{*}$-continuous convolution
semigroup of $\Rant^{\mathrm{u}}$-invariant states of $\CzU{\G}$
with unbounded generator (and vice versa) \citep[Theorem 4.6]{Skalski_Viselter__convolution_semigroups}.
Let us establish the converse in the more general framework of this
section.
\begin{thm}
If $\G$ is a locally compact quantum group and $\gamma$ satisfies
the assumptions in \prettyref{thm:conv_smgrp_from_CND_func}
and is unbounded (equivalently: unbounded after restricting to $\CzU \G$), then $\widehat{\G}$ does not have Property~(T).
\end{thm}

\begin{proof}
Applying \prettyref{thm:conv_smgrp_from_CND_func} we get a $w^{*}$-continuous
convolution semigroup $\left(\mu_{t}\right)_{t\ge0}$ of $\Rant^{\mathrm{u}}$-invariant
contractive positive functionals of $\CzU{\G}$ whose generating functional $\gamma'$ extends
$\gamma$ on $\CzU \G$. Since $\gamma$ is unbounded on $\CzU \G$, $\gamma'$ is unbounded. This means that
$\left(\mu_{t}\right)_{t\ge0}$ is not norm continuous \citep[Theorem 3.7]{Lindsay_Skalski__conv_semigrp_states}.
By normalising, we can assume that $\left(\mu_{t}\right)_{t\ge0}$ consists of states.
This implies that $\widehat{\G}$ does not have Property~(T) by \citep[Theorem 6.1]{Daws_Skalski_Viselter__prop_T}.
\end{proof}

\section{Example: compact quantum groups} \label{sec:cqg_examples}

In this section we show that \prettyref{thm:conv_smgrp_from_CND_func}
can be applied to prove Sch\"urmann's reconstruction theorem for
compact quantum groups assuming that the functional is $\Rant^{\mathrm{u}}$-invariant.
Our proof is analytic, and is very different from the original one. It is worth noting that a proof of the Sch\"urmann reconstruction theorem for compact quantum groups using Dirichlet form techniques was circulated a few years ago in unpublished notes of Roland Vergnioux.

Let $\G$ be a compact
quantum group. For $S\subseteq\Irred{\G}$ write $\Pol{\G}_{S}:=\linspan\{ u_{ij}^{\a}:\a\in S,1\le i,j\le n_{\a}\} $.
Also let $\Pol{\G}_{\a}:=\Pol{\G}_{\left\{ \a\right\} }$ for $\a\in\Irred{\G}$.
Denote by $h$ the Haar state of $\G$, both on $\CU{\G}$ and on
$\Linfty{\G}$. Recall the orthogonality relation 
\[
h(u_{ij}^{\a*}u_{kl}^{\be})=\frac{1}{tr(Q_{\a})}\delta_{\a\be}\delta_{jl}(Q_{\a}^{-1})_{ki}\qquad(\forall_{\a,\be\in\Irred{\G}}\forall_{1\le i,j\le n_{\a},1\le k,l\le n_{\be}})
\]
for suitable invertible positive matrices $Q_{\a} \in \matrices_{n_\a}$, $\a \in \Irred{\G}$.
Let $\a\in\Irred{\G}$. For $1 \le s,t \le n_\a$, the functional on either $\CU{\G}$ or $\Linfty{\G}$
given by 
\[
x\mapsto\sum_{i=1}^{n_{\a}}tr(Q_{\a})\cdot\left(Q_{\a}\right)_{is}h(u_{it}^{\a*}x)
\]
maps $u_{kl}^{\be}$ to $\delta_{\a\be}\delta_{ks}\delta_{lt}$ for all $\be\in\Irred{\G}$
and $1\le k,l\le n_{\be}$. These functionals yield a bounded linear map $P_{\a}$
on $\CU{\G}$ acting as the identity on $\Pol{\G}_{\a}$ and annihilating
$\Pol{\G}_{\Irred{\G}\backslash\left\{ \a\right\} }$, and a similar
(normal) map exists on $\Linfty{\G}$ (for more information on such maps in the broader context of compact quantum group actions we refer to \cite[Section 3]{DeCommer_Notes}). Writing $\gnsmap$ for the GNS map of $h$ and $p_{\a}$ for the (orthogonal)
projection of $\Ltwo{\G}$ onto $\Ltwo{\G}_{\a} := \gnsmap(\Pol{\G}_{\a})$, we clearly have
$p_{\a}\circ\gnsmap=\gnsmap\circ P_{\a}$. Note that $P_{\a}$
commutes with the scaling group of $\G$. Furthermore, for all $\z\in\Ltwo{\G}$,
we have $P_{\a}((\widehat{\om}_{\z}\tensor\i)(\widehat{\wW}))=(\widehat{\om}_{p_{\overline{\a}}\z}\tensor\i)(\widehat{\wW})$,
and consequently
\begin{equation}
P_{\a}(\aa{\widehat{\om}_{\z}})=\aa{\widehat{\om}_{p_{\overline{\a}}\z}}.\label{eq:CQG_proj_a}
\end{equation}
Finally, for $S\subseteq\Irred{\G}$, set $p_S := \sum_{\a\in S} p_\a$, and if $S$ is finite, set $P_S := \sum_{\a\in S} P_\a$.
\begin{lem}
\label{lem:CQG_a_proj}For every $\eta\in\Ltwo{\G}$ and $S\subseteq\Irred{\G}$
we have 
\[
\left[(\i+\Rant^{\mathrm{u}})(\epsilon(\cdot)\one-\i)\right](\aa{\widehat{\om}_{p_{S}\eta}})\le\left[(\i+\Rant^{\mathrm{u}})(\epsilon(\cdot)\one-\i)\right](\aa{\widehat{\om}_{\eta}}).
\]
\end{lem}

\begin{proof}
We follow the line of proof of \prettyref{lem:id_plus_R__e_minus_i__ineq}.
For every $\Rant^{\mathrm{u}}$-invariant state $\nu$ of $\CU{\G}$,
since $\one-\widetilde{R}_{\nu}^{(2,\varphi)}$ is positive and belongs
to $\linfty{\widehat{\G}}$, we have $\widehat{\omega}_{p_{S}\eta}(\one-\widetilde{R}_{\nu}^{(2,\varphi)})\le\widehat{\omega}_{\eta}(\one-\widetilde{R}_{\nu}^{(2,\varphi)})$,
which is equivalent to $(\epsilon-\nu)(\aa{\widehat{\om}_{p_{S}\eta}})\le(\epsilon-\nu)(\aa{\widehat{\om}_{\eta}})$.
From this one readily infers the desired inequality.
\end{proof}
\begin{thm}
Let $\G$ be a compact quantum group and $\gamma\colon \Pol{\G}\to\C$ be
a linear functional satisfying $\gamma(\one)=0$ that is $\Rant^{\mathrm{u}}$-invariant
and conditionally positive. Then $\gamma$ satisfies the assumptions
of \prettyref{thm:conv_smgrp_from_CND_func}.
\end{thm}

\begin{proof}
\prettyref{enu:conv_smgrp_from_CND_func__1} and \prettyref{enu:conv_smgrp_from_CND_func__2}
are clear (see \prettyref{exa:our_algebra_compact_case}, and note that $\aa{\widehat{\om}_{\z}} \in \Pol \G$ if and only if $\z$ belongs to the \emph{algebraic} direct sum of the subspaces ${\Ltwo\G}_\a$, $\a \in \Irred \G$).

\prettyref{enu:conv_smgrp_from_CND_func__3} Let a sequence $\left(a_{k}\right)_{k=1}^\infty$
in $\mathscr{D}_{+}\cap\Pol{\G}$ converge to some $a\in\mathscr{D}_{+}\cap\Pol{\G}$.
Let $F$ be a finite subset of $\Irred{\G}$ such that $a\in\Pol{\G}_{F}$.
For each $k\in\N$, the elements $a_{k}^{1}:=P_{F}(a_{k})$
and $a_{k}^{2}:=a_{k}-a_{k}^{1}$ satisfy $a_{k}=a_{k}^{1}+a_{k}^{2}$,
$a_{k}^{1},a_{k}^{2}\in\mathscr{D}_{+}\cap\Pol{\G}$ by \prettyref{eq:CQG_proj_a}, $a_{k}^{1}\in\Pol{\G}_{F}$
and $a_{k}^{2}\in\Pol{\G}_{\Irred{\G}\backslash F}$. Clearly $a_{k}^{1}\xrightarrow[k\to\infty]{}a$.
From linearity of $\gamma$ and finite dimensionality of $\Pol{\G}_{F}$
we deduce that $\gamma(a_{k}^{1})\xrightarrow[k\to\infty]{}\gamma(a)$.
Since $-\gamma(a_{k}^{1}),-\gamma(a_{k}^{2})\ge0$ for every $k\in\N$
by \prettyref{cor:general_gen_func__D_plus}~\prettyref{enu:general_gen_func__D_plus__1},
the inequality $-\gamma(a)\le\liminf_{k\to\infty}(-\gamma(a_{k}))$
is now obvious.

\prettyref{enu:conv_smgrp_from_CND_func__4} Let $\mathcal{F}$ denote
the set of all finite subsets of $\Irred{\G}$ directed by inclusion.
For $\eta\in\Ltwo{\G}$, the net $\left(p_{S}\eta\right)_{S\in\mathcal{F}}$
converges to $\eta$, and for each $S\in\mathcal{F}$ we have $\aa{\widehat{\om}_{p_{S}\eta}}\in\Pol{\G}$.
It follows from \prettyref{lem:CQG_a_proj} that condition~\prettyref{enu:conv_smgrp_from_CND_func__4_b}
of \prettyref{cor:conv_smgrp_from_CND_func} holds.
\end{proof}

\section*{Funding}
The first author was partially supported by the National Science Centre (NCN) \linebreak{} [2014/14/E/ST1/00525]. 

\section*{Acknowledgements}

We thank V.~Runde, N.~Spronk and L.~Turowska for helpful correspondences
about the content of this paper. We are grateful to P.~Ohrysko for
communicating to us his proof that $\CzX{\infty}{\R}\nsubseteq B(\R)$
and letting us present it here (see \prettyref{exa:our_algebra_classical_case}).
Some of the work on this paper was done during a visit of the first author to Haifa in April 2018, and during a visit of the second author to Warsaw in September 2018; the hospitalities of the mathematics department/institute are gratefully acknowledged by both authors. We thank also the referees for their careful reading of the initial version of the paper and several comments improving the presentation and content of the text.

\bibliographystyle{amsalpha}
\bibliography{GenFunc2}

\end{document}